\documentclass[11pt]{article}

\usepackage[breaklinks, colorlinks, linkcolor=blue, citecolor=blue, urlcolor=black]{hyperref}

\usepackage{amsmath}
\usepackage{amssymb}
\usepackage{amsthm}

\usepackage[utf8]{inputenc}
\usepackage{graphicx}
\usepackage{todonotes}
\usepackage{tabularx}
\usepackage{tabularcalc}
\usepackage{booktabs}
\usepackage{multirow}
\usepackage{algorithm}
\usepackage{algorithmic}
\usepackage{subfig}
\usepackage{siunitx}
\usepackage{soul}
\usepackage{footmisc}
\usepackage{svg}
\usepackage{zibtitlepage}
\usepackage{xspace}

\newtheorem{Def}{Definition}
\newtheorem{Lem}{Lemma}
\newtheorem{Thm}[Lem]{Theorem}
\newtheorem{Cor}[Lem]{Corollary}

\newtheorem{Asm}{Assumption}

\newcommand{\ie}{i.e.,\xspace}
\newcommand{\eg}{e.g.,\xspace}

\newcommand{\wlogU}{W.l.o.g.\ }

\newcommand{\bandb}{branch-and-bound\xspace}

\newcommand{\rhs}{right-hand side\xspace}
\newcommand{\lhs}{left-hand side\xspace}

\newcommand{\NP}{\mathcal{NP}}

\newcommand{\Abs}[1]{\left| #1\right|}

\newcommand{\down}[1]{\lfloor #1\rfloor}
\newcommand{\up}[1]{\lceil #1\rceil}
\newcommand{\fpdown}[1]{\underline{#1}}
\newcommand{\fpup}[1]{\overline{#1}}

\DeclareMathOperator*{\argmin}{arg\,min}

\newcommand{\solver}[1]{\textsc{#1}\xspace}

\newcommand{\scipversion}{8.0}

\newcommand{\scip}{\solver{SCIP}}
\newcommand{\scipv}{\solver{SCIP}~\scipversion\xspace}

\newcommand{\soplexversion}{6.0.2}
\newcommand{\soplexv}{\solver{SoPlex}~\soplexversion\xspace}

\newcommand{\papiloversion}{2.0.1}
\newcommand{\papilov}{\solver{PaPILO}~\papiloversion\xspace}

\newcommand{\miplib}{\mbox{MIPLIB}}

\newcommand{\N}{\mathbb{N}\xspace}

\newcommand{\Z}{\mathbb{Z}\xspace}
\newcommand{\R}{\mathbb{R}\xspace}
\newcommand{\Q}{\mathbb{Q}\xspace}

\newcommand{\F}{\mathbb{F}\xspace}

\newcommand{\overhead}[2]{\pgfmathparse{(#1-#2)/#2 *100}\pgfmathprintnumber[precision=1]{\pgfmathresult}\%}
\newcommand{\reduction}[2]{\pgfmathparse{(#2-#1)/#2 *100}\pgfmathprintnumber[precision=1]{\pgfmathresult}\%}
\newcommand{\increase}[2]{\pgfmathparse{(#1-#2)/#2 *100}\pgfmathprintnumber[precision=1]{\pgfmathresult}\%}

\newcommand{\colorquot}[2]{
   \pgfmathparse{(#1<0.9*#2)?1:0}\ifdim\pgfmathresult pt>0pt \textcolor{blue}{$\mathbf{\reduction{#1}{#2}}$}\else
      \pgfmathparse{(#1>1.1*#2)?1:0}\ifdim\pgfmathresult pt>0pt \textcolor{red}{$\mathbf{\reduction{#1}{#2}}$} \else
         $\reduction{#1}{#2}$\fi \fi
}

\newcommand{\nocuts}{\textsc{nocuts}\xspace}
\newcommand{\safegmi}{\textsc{safegmi}\xspace}
\newcommand{\enc}{\textsc{densmall}\xspace}
\newcommand{\fpred}{\textsc{fpnocuts}\xspace}
\newcommand{\fpcuts}{\textsc{fpgmi}\xspace}

\usepackage{longtable}
\usepackage{chngcntr}

\newcommand{\feasP}{$P$\xspace}

\newlength\myindent
\newcommand\bindent{\begingroup
   \setlength{\itemindent}{\myindent}
   \addtolength{\algorithmicindent}{\myindent}
}
\newcommand\eindent{\endgroup}

\svgpath{{pictures/}} 

\newcommand{\myorcidlink}[1]{\,\href{https://orcid.org/#1}{\raisebox{-0.45ex}{\includegraphics[width=1.8ex]{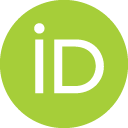}}}}

\begin{document}

\ZTPTitle{Safe and Verified Gomory Mixed Integer Cuts in a Rational MIP Framework}
\ZTPAuthor{ \ZTPHasOrcid{Leon Eifler}{0000-0003-0245-9344} \ZTPHasOrcid{Ambros Gleixner}{0000-0003-0391-5903}
}
\ZTPNumber{ZIB-Report 23-09}
\ZTPMonth{March}
\ZTPYear{2023}

\newpage

\title{Safe and Verified Gomory Mixed Integer Cuts in a Rational MIP Framework
\thanks{The work for this article has been conducted within the Research Campus Modal funded by the German Federal Ministry of Education and Research (BMBF grant numbers 05M14ZAM, 05M20ZBM).}}
\author{Leon Eifler$^1$\myorcidlink{0000-0003-0245-9344} \and
   Ambros Gleixner$^{1,2}$\myorcidlink{0000-0003-0391-5903}
}
\date{$^1$Zuse Institute Berlin\\$^2$HTW Berlin\\[2ex]\today
}

\maketitle

\begin{abstract}
   This paper is concerned with the exact solution of mixed-integer programs (MIPs) over
   the rational numbers, \ie without any roundoff errors and error tolerances. Here, one computational bottleneck that should be avoided whenever possible is to employ large-scale symbolic computations. Instead it is often possible to use safe directed rounding methods, \eg to generate provably correct dual bounds.
   In this work, we continue to leverage this paradigm and extend an exact branch-and-bound framework by separation routines for safe cutting planes, based on the approach first introduced by Cook, Dash, Fukasawa, and Goycoolea in 2009. Constraints are aggregated safely using approximate dual multipliers from an LP solve, followed by mixed-integer rounding to generate provably valid, although slightly weaker inequalities.

   We generalize this approach to problem data that is not representable in floating-point arithmetic, add routines for controlling the encoding length of the resulting cutting planes, and
   show how these cutting planes can be verified according to the VIPR certificate standard.
   Furthermore, we analyze the performance impact of these cutting planes in the context of an exact MIP framework, showing that we can solve $21.5\%$ more instances and reduce solving times by $26.8\%$ on the \miplib~2017 benchmark test set.
\end{abstract}

\section{Introduction}

Even though the problem class of mixed-integer programs (MIPs) is $\NP$-hard \cite{Conforti2014}, state-of-the-art MIP solvers manage to solve a large number of such problems with up to millions of variables and constraints
\cite{GleixnerHendelGamrathetal.2021}.
While the core algorithm, LP-based \bandb, is straightforward, it is a long list of sophisticated solving techniques such as presolving, cutting planes, primal heuristics, conflict analysis, and branching rules that make this remarkable performance possible.

However, virtually all established solvers that contain these techniques rely on fast floating-point arithmetic, combined with numerical error tolerances to achieve a high degree of numerical stability.
For most applications, especially in industry, this is completely sufficient.
Nevertheless, some problem cases require exact proofs of optimality or infeasibility without the slight numerical inaccuracies that result from rounding errors in floating-point arithmetic.
This is the case when mixed-integer programs are used as
a tool in mathematics, \eg to computationally investigate open conjectures.
Recent examples of such approaches include
\cite{Bofi19,Burt12,EiflerGleixnerPulaj2022,KenterEtAl2018,LanciaEtAl2020,Pula20}.

Examples for industrial applications where the correctness of results is
paramount are, \eg hardware verification~\cite{Achterberg2007}, compiler
optimization~\cite{WilkenEtAl2000}, or more recently infeasibility analysis in hydro unit commitment~\cite{SahraouiBendottiAmbrosio2019}.

To the best of our knowledge, the first fully general, exact MIP solver that can handle MIPs with rational input data is presented by Cook, Koch, Steffy, and Wolter~\cite{CookKochSteffyetal.2013}.
Their hybrid-precision framework uses both symbolic as well as numeric computations
and applies different dual bounding methods~\cite{Espinoza2006,Neumaier02safebounds,SteffyWolter2013} to generate provably valid dual bounds.
However, besides employing reliability pseudocost branching, this framework still lacks the previously mentioned advanced solving techniques since their application in the roundoff-error free setting is often not trivial.

While a direct translation of methods using symbolic computations is always possible, it is often prohibitively slow in practice. Safe methods that try to avoid symbolic computations in favour of safe rounding techniques can often provide better results. A first step in the direction of closing this algorithmic gap between exact and inexact MIP was established \cite{EiflerGleixner2022} by revising the approach of  Cook et al.~\cite{CookKochSteffyetal.2013} and extending it by symbolic presolving routines, as well as a
repair step that enables the usage of all existing primal heuristics \cite{EiflerGleixner2022}.

One key feature missing to date is a separation routine for cutting planes, which is known to be among
the most important components to make MIP solvers perform well in practice, as reported, \eg by Achterberg and Wunderling \cite{AchterbergWunderling2013}.
Among different types of cutting planes, they found that especially mixed-integer rounding (MIR) cuts ($48\%$ speedup), as well as Gomory cuts ($28\%$ speedup), seem to provide the most benefit.
However, designing an efficient roundoff-error-free separation procedure for Gomory cuts is non-trivial since a purely symbolic approach would require an exact LP tableau row. To compute the tableau, the LP relaxation needs to be solved exactly and LP rows need to be aggregated, both of which are computationally expensive operations.
Instead, Cook, Dash, Fukasawa, and Goycoolea \cite{CookDashFukasawaGoycoolea2009} introduced a technique that can be used to construct numerically safe cuts, using safe rounding techniques. These cuts are guaranteed to be exactly feasible, without any symbolic computations, but at the cost of slightly weaker efficacy.
While the focus of Cook et al. was to provide safe cutting planes within a normal floating-point MIP solver, the technique can be generalized to the exact MIP setting.

This generalization is our first contribution. We show how to relax rational problem data to be usable for the safe rounding technique and generalize the approach of \cite{CookDashFukasawaGoycoolea2009} to allow for negative variable bounds. Our second contribution is adapted scaling and a new post-processing technique that improves the performance of safe cutting planes in the exact MIP framework. 
Finally, our third major contribution is the independent verification of these cutting planes using the VIPR \cite{VIPR} certificate format.

To motivate the last point further, note that although the exact MIP framework guarantees an exactly optimal solution \emph{in theory}, the implemented computer program is very complex and a user might not have full confidence that the result is correct \emph{in practice}.
Our goal is for the solving algorithm to provide a certificate of optimality for each instance it solves which can be verified independently. Such proof-logging features are standard in the world of satisfiability solvers~\cite{HeuleHuntWetzler2013, HeuleHuntWetzler2014} and have also been adapted to pseudo-Boolean~\cite{ElffersGochtMcCreeshNordstrom_2020, stephan_gocht_2019_3548582} and maximum satisfiability solvers~\cite{BergBogaertsNordstromOertelVandesande2023}. In the context of {integer programming, certifaction has been employed in the past for the travelling salesman problem \cite{Applegate2009}, to verify an optimal tour with 85,900 cities by a problem-specific branch-and-cut certificate. For general purpose exact MIP,} the VIPR certificate standard \cite{VIPR} is a possibility to encode a proof for the optimality of a \bandb tree, using only elementary branching logic and basic arithmetic. We introduce an explicit verification of MIR cuts using only the logic of VIPR, as well as an algorithm to account for the rounding errors introduced through the safe rounding procedure.

The overall goal of this research is twofold. First, we aim to further reduce the algorithmic gap between exact and floating-point MIP solvers to make exact MIP solving a real alternative for researchers going forward. Second, we aim to better understand the practical differences between numerically safe cutting planes in exact MIP and normal ones in conventional floating-point MIP, especially the impact they have on LP solving behaviour. We conduct a detailed computational study, investigating both of these points.

The paper is structured as follows. In Section~\ref{sec:safe-cuts}, we present the construction of numerically safe Gomory mixed-integer cuts in an exact MIP setting, discuss scaling, and introduce a new post-processing 
routine that controls the encoding length of cut coefficients. In Section~\ref{sec:veri-cuts}, we present verification routines for MIR cuts and a completion algorithm that accounts for the rounding errors introduced during safe rounding. Section~\ref{sec:comp} constitutes our computational study, where the performance of all new features is evaluated and compared against an analogue experiment within a floating-point solver. We conclude with remarks on further research directions in Section~\ref{sec:conc}.

All the code for completing and checking the certificates is freely available on GitHub \cite{VIPRweb}.

\section{Safe Gomory mixed-integer cuts for exact MIP}
\label{sec:safe-cuts}

We consider the general rational mixed-integer program (MIP) of the form

\[
   \min\{\, c^Tx : Ax \leq b,\; \ell \leq x \leq u,\text{ and } x_j\in\Z \text{ for all } j \in I \,\}
\]
where $A \in \Q^{m \times n}, c,\ell,u \in \Q^n, b \in \Q^m$ and $I \subseteq \{1,\ldots,n\}$. We denote the set of all feasible points for this MIP by \feasP.
The variable bounds are allowed to be $\pm\infty$, but we assume throughout this paper that all variables appearing in a cut have either a finite lower or a finite upper bound:
\begin{Asm}
   \label{asm:bounds}
   For any valid inequality $a^Tx \le b$ encountered during the construction of a cut, if $a_i \neq 0$ then either $u_i < \infty$ or $\ell_i > -\infty$ holds. Thus we can choose sets $U,L$ such that $u_i < \infty$ for all $i \in U$, $\ell_i > -\infty$ for all $i \in L$, and $a_i = 0$ for all $i \notin U \cup L$.
\end{Asm}

{Note that this assumption can always be fulfilled by splitting variables with infinite bounds into their positive and negative part; however, it is necessary to introduce new auxiliary variables explicitly for these parts. In practice, after presolving, Assumption~\ref{asm:bounds} usually seems to hold and we did not encounter any problems where cuts needed to be discarded in our experiments.}

\noindent
To give an overview, our steps for generating safe MIR cuts in general and safe
Gomory mixed-integer cuts in particular are as follows:
\begin{enumerate}
   \item \emph{Approximate} rows by a floating-point representable relaxation.
   \item Generate a valid, floating-point representable inequality for \feasP using \emph{safe aggregation} of rows.
   \item Construct the MIR cut:
         \begin{enumerate}
            \item Safely \emph{transform} this inequality into non-negative variable space.
            \item Apply the \emph{safe MIR technique} to the transformed inequality and retransform to original variable space.
            \item \emph{Substitute} slack variables.
         \end{enumerate}
   \item Post-process the cut:
         \begin{enumerate}
            \item \emph{Scale} the resulting cut to improve numerical stability and possibly make coefficients integer.
            \item Control the \emph{encoding length} of the generated cut.
         \end{enumerate}
\end{enumerate}
Steps 2 and 3 are in essence the procedure described in \cite{CookDashFukasawaGoycoolea2009},
although slight differences exist because we allow for variables with negative lower bounds in $\ell$.
In the following sections, we describe all steps in more detail.

\subsection{Safe MIR cuts by directed rounding}

Let $\alpha^Tx \le \beta$ be a valid inequality for \feasP. If all integer variables are non-negative, then the mixed-integer rounding cut
\begin{equation}
   \label{eq:mircut1}
   \sum_{i \in I} \Big( \down{\alpha_i} + \frac{(f_i -f)^+}{1-f} \Big) x_i + \sum_{i \notin I, a_i < 0} \frac{\alpha_i
   }{1-f}x_i\le \down{\beta}
\end{equation}
with $f = \beta - \down{\beta}$ and $f_i = \alpha_i - \down{\alpha_i}, i \in I$, is valid for \feasP~\cite{MarchandWolsey1998}.
If not all integer variables are non-negative but Assumption~\ref{asm:bounds} holds, then we can transform all integer variables to a non-negative space via
\begin{align}
   \label{eq:trans}
   x_i' := \begin{cases}
              u_i - x_i \text{ for all } i \in U \cap I, \\
              x_i - \ell_i \ \text{ for all } i \in L \cap I.
           \end{cases}
\end{align}
Gomory mixed-integer (GMI) cuts are obtained by applying the MIR technique to one row of an optimal simplex tableau that corresponds to an integer variable with a fractional LP solution value.

The na\"ive approach to computing exact GMI cuts for \feasP by roundoff-error-free rational arithmetic can become prohibitively expensive not only due to the operations involved in \eqref{eq:mircut1}.
Most of all, obtaining $\alpha^Tx \le \beta$ as an exact tableau row requires a rational LU factorization of the basis matrix.
In recent approaches to solve MIPs exactly, however, only an approximate floating-point LP solution is at hand \cite{EiflerGleixner2022}.

Instead of performing all computations in exact arithmetic, a better option is to construct numerically safe cuts as proposed by \cite{CookDashFukasawaGoycoolea2009}. By using safe directed rounding, valid inequalities can be generated that are guaranteed to be exactly feasible for \feasP, without symbolic computations, though at the cost of obtaining slightly weaker cuts.
Hence, let us first define the notation used for safe rounding that will be needed throughout this paper.

Let $\F \subseteq \Q$ denote the set of floating-point numbers. In practice, these will usually be standard IEEE double-precision numbers with $11$~bits for the exponent and $52$~bits for the mantissa.
{This means any number $f \in \F$ can be written as $f = s * m * 2^{e}$, where the sign is $s \in \{-1,1\}$, the mantissa $m=(1,M)_2$ is a number between $1$ and $2$, {encoded in binary format} with $52$ bits, and the exponent $e$ is a binary number with up to $11$ bits, centered around $0$, \ie $e \in [-1022,1023]$. For a detailed description of the format, see \eg \cite{Orton2001}. It is clear that not all rational numbers can be exactly represented in this format and that $\F$ is not closed under addition, subtraction, multiplication, or division. Hence, the result of these operations is rounded up or down to a nearest representable number; the rounding direction can be controlled in computer code.}

\begin{Def}
   Let $x \in \Q$ be a rational number. We denote the closest upper and lower floating-point representable approximations of $x$ in $\F$ by
   \[
      \fpup{x} := \min\{ y \in\F : y \ge x \}
      \qquad\text{and}\qquad
      \fpdown{x} := \max\{ y \in\F : y \le x \},
   \]
   respectively.  We call $x$ \emph{$\F$-representable} if $x\in\F$.
   We call an inequality $a^Tx \le b$ \emph{$\F$-representable} if $a \in \F^n$ and $b \in \F$.
   Finally, for $n \ge 2$ and $\lambda_1,\ldots,\lambda_n\in\F$, we define recursively
   \begin{align*}
      \fpup{\sum_{i=1}^n \lambda_i} := \fpup{\lambda_1  + \fpup{\sum_{i=2}^n \lambda_i}}
      \qquad\text{and}\qquad
      \fpdown{\sum_{i=1}^n \lambda_i} := \fpdown{\lambda_1  + \fpdown{\sum_{i=2}^n \lambda_i}}.
   \end{align*}
\end{Def}
{We handle other arithmetic operations, \ie subtraction, multiplication, and division, as well as combinations thereof, analogously. We note that this is consistent with how modern computers handle floating-point numbers. Specifically, computers have the following rounding modes: down, up, towards zero, and to nearest. These can be changed by the algorithm to ensure the desired over- or underestimation.}

Note that this definition means that no rational computations are necessary to compute $\fpup{\sum_{i=1}^n \lambda_i}$. On the other hand, the order of summation changes the result since we round after each pairwise addition.

\subsection{Approximating problem data}

Some coefficients in rows of the original problem formulation may not be in $\F$.
Affected inequalities can be made $\F$-representable by relaxing them slightly
according to the following formula, which relies on Assumption~\ref{asm:bounds}.

\begin{Lem}
   \label{lem:fprep}
   Let $a^Tx \le b$, with $a \in \Q^n, b \in \Q$ be valid for \feasP. Then the $\F$-representable inequality
   \begin{align*}
      \sum_{i \in U} \fpup{a_i} x_i + \sum_{i \in L} \fpdown{a_i} x_i \le \fpup{b + \sum_{{i \in U}} (\fpup{a_i}-\fpdown{a_i}){\fpup{u_i}} + \sum_{{i \in L}}(\fpdown{a_i}-\fpup{a_i}){\fpdown{\ell_i}}}
   \end{align*}
   is also valid for \feasP.
\end{Lem}
\begin{proof}
   The first step is to transform all variables into a non-negative form. Define
   $x_i'$ according to \eqref{eq:trans} for all $i \in I$.
   Substituting $x$ for $x'$ yields
   \begin{align*}
      \sum_{i \in U}a_i(u_i-x_i')+\sum_{i \in L}a_i(\ell_i+x_i') & \le b.
   \end{align*}
   Moving all constants to the \rhs
   \begin{align*}
      \sum_{i \in U}-a_ix_i'+\sum_{i \in L}a_ix_i' & \le b - \sum_{i \in U} a_iu_i - \sum_{i \in L}a_i\ell_i.
   \end{align*}
   Since $x_i' \ge 0$, we can round down all coefficients on the \lhs.
   With $\fpdown{-a_i}=-\fpup{a_i}$, we get
   \begin{align*}
      \sum_{i \in U}-\fpup{a_i}x_i'+\sum_{i \in L}\fpdown{a_i}x_i' & \le b - \sum_{i \in U} a_iu_i - \sum_{i \in L}a_i\ell_i.
   \end{align*}
   Substituting $x$ for $x'$ yields
   \begin{align*}
      \sum_{i \in U} \fpup{a_i} x_i + \sum_{i \in L} \fpdown{a_i} x_i \le {b + \sum_{i \in U} (\fpup{a_i}-a_i)u_i + \sum_{i \in L}(\fpdown{a_i}-a_i)\ell_i}.
   \end{align*}
   Finally, rounding the whole \rhs upwards, we get the final inequality.
\end{proof}

This procedure of implicitly transforming to non-negative variable space, rounding in the correct direction, retransforming, and finally rounding again is the essential technique from \cite{CookDashFukasawaGoycoolea2009} that is used for all the safe operations in this manuscript. {Since the correction factor on the right-hand side of the inequalities always has non-negative coefficients for $i \in U$ and non-positive coefficients for $i \in L$, we always use $\fpup{u}$ or $\fpdown{\ell}$, if the bounds are not $\F$-representable. To avoid notational clutter, we will omit the bars on the bounds for the remainder of the manuscript.}

\subsection{Safe aggregation of rows}
The technique from the previous Lemma can also be applied to aggregate two $\F$-representable rows safely.

\begin{Cor}
   \label{cor:agg}
   Let $a^Tx \le b$ and $c^Tx \le d$ be two valid, $\F$-representable inequalities, and $0 < \lambda \in \F$.
   If Assumption~\ref{asm:bounds} holds, we can generate an $\F$-representable, valid approximation of the aggregated inequality $(a+\lambda c)^T x \le b + \lambda d$ by
   \begin{equation*}
      \sum_{i \in U}\overline{\alpha_i}x_i+\sum_{i \in L}\underline{\alpha_i}x_i \le \fpup{b+\lambda d + \sum_{i \in U, u_i > 0} (\overline{\alpha_i}-\underline{\alpha_i})u_i + \sum_{i \in L, \ell_i < 0}(\underline{\alpha_i}-\overline{\alpha_i})\ell_i},
   \end{equation*}
   with $\alpha_i := a_i + c_i\lambda$.
\end{Cor}
\begin{proof}
   The proof follows the same steps as in Lemma~\ref{lem:fprep}.
\end{proof}

The rows of an optimal simplex tableau, which form the base inequalities
$\alpha^T x \leq \beta$ to derive GMI cuts via the MIR
formula~\eqref{eq:mircut1}, can be alternatively obtained as an aggregation of
inequalities from the LP relaxation.
As a matter of fact, this is also how GMI cuts are computed in floating-point MIP
solvers to minimize numerical errors.
Similarly, we can perform safe aggregation of problem inequalities according to Corollary~\ref{cor:agg}, using multipliers from an \emph{approximate} row of the basis inverse to generate a valid, $\F$-representable base inequality $\alpha^T x \leq \beta$.
This is hopefully a good approximation to the exact simplex tableau row, at
least if the approximate floating-point LP solution from which we obtain the
dual multipliers is a good approximation to the optimum of the exact, rational
LP relaxation.

\subsection{Constructing the MIR cut}
\label{subsec:mir}

After obtaining such a row $\alpha^Tx \le \beta$, we proceed as we normally would in deriving a MIR cutting plane. In the following, we go through those steps again to highlight at which points operations have to be performed in a numerically safe fashion.

First, for each variable one of the bounds is chosen to transform to non-negative variable space, since the MIR formula requires non-negative variables. If the variable has both finite upper and lower bound, we use the bound that is closest to the value of the current LP solution for that variable. Denote by $U$ the index set of variables for which the upper bound is chosen, by $L$ the set for which the lower bound is chosen, and by $x_i'$ the transformed variables according to \eqref{eq:trans}.

We first obtain the {aggregated} inequality in transformed space
\begin{align}
   \sum_{i \in U} -\alpha_i x_i' + \sum_{i \in L} \alpha_i x_i' \le \overline{\beta - \sum_{i \in U}\alpha_iu_i-\sum_{i \in L} \alpha_i \ell_i} := d.
\end{align}
Next, we compute the coefficients in the cut, according to {\eqref{eq:mircut1}}. Denote the transformed coefficients by
\begin{align}
   \alpha_i' = \begin{cases}
                  \alpha_i & \text{ for } i \in L \\ -\alpha_i &\text{ for } i \in U,
               \end{cases}
\end{align}
and the fractionalities by $f = d - \down{d}$, $f_i = \alpha_i' - \down{\alpha_i'}$. Then the safe MIR cut in the transformed space is

\begin{equation}
   \sum_{i \in I} \underline{\Big( \down{\alpha_i'} + \frac{(f_i -f)^+}{1-f} \Big)} \,x_i' + \sum_{i \notin I:\alpha_i' < 0} \underline{\Big( \frac{\alpha_i'
      }{1-f}\Big)}\,x_i'\le \down{d}.
\end{equation}
Transforming back to original variable space as in Lemma~\ref{lem:fprep} yields the final safe MIR cut.

We give some additional, slightly technical details that will be especially relevant for the verification described in Section~\ref{sec:veri-cuts}.
Whenever we consider a row \( a_i^T x \leq b_i \) for aggregation, we implicitly turn it into an equation using a slack variable, \ie \( a_i^T x+s_i=b_i, s_i \ge 0 \). This allows two generalizations. Firstly, we can aggregate with a multiplier of any sign. If we use a negative multiplier $\lambda_i$ {for a row without additional integrality information}, the slack is treated as a continuous variable and thus gets assigned the coefficient \( \frac{\lambda_i}{(1-f)} \), where $f$ is the fractionality of the one-row relaxation's \rhs. Secondly, and more importantly, if we know that all non-zero coefficients as well as their corresponding variables are integer, then $s$ can be treated as an integer variable for the MIR procedure and gets assigned the coefficient \( \lfloor\lambda_i\rfloor+\frac{\left(f_i-f\right)^+}{(1-f)} \), {regardless of the sign of $\lambda$}.

At the very end, the slack is eliminated using its definition \( s_i=b_i-a_i^{\top} x \) to return to the space of original variables. Since $s_i$ is by definition non-negative, it is straightforward to do this safely using directed rounding. We simply need to underapproximate the coefficient and perform the back-substitution exactly as an aggregation according to Corollary~\ref{cor:agg}.

\subsection{Post-processing steps}

The steps outlined in the previous section yield a new valid, $\F$-representable inequality $\alpha^Tx \le \beta$. We present two post-processing ideas that aim to improve LP performance after adding such inequalities.

\subsubsection{Scaling}

An important aspect to improve the numerical stability of {cutting plane algorithms}, both in the exact and in the inexact setting, is scaling. Large coefficient ranges in the problem are known to detriment the performance and accuracy of LP solvers.
This is especially relevant for the exact MIP setting.
Given a cut $\alpha^Tx \le \beta$ and a scaling coefficient $s \ge 0$, we can safely scale the cut using the same approach as in the previous subsection.
\begin{Lem}
   Given a valid
   cut $\alpha^Tx \le \beta$ and a scaling factor $s \ge 0$, the safely scaled cut
   \begin{equation*}
      \sum_{i \in U} \fpup{s \alpha_i} x_i + \sum_{i \in L} \fpdown{s \alpha_i} x_i \le \fpup{s \beta + \sum_{i \in U, u_i > 0} (\fpup{s \alpha_i}-\fpdown{s \alpha_i})u_i + \sum_{i \in L, \ell_i < 0}(\fpdown{s \alpha_i}-\fpup{s \alpha_i})\ell_i}
   \end{equation*}
   is valid and $\F$-representable.
\end{Lem}
\begin{proof}
   Analogous to Lemma~\ref{lem:fprep}.
\end{proof}

\noindent
We employ two different scaling approaches, similar to the ones that are used by the corresponding floating-point algorithm in \scip. If continuous variables appear in the cut, we simply scale to \emph{equilibrium}, meaning we scale such that the largest absolute value of any coefficient becomes $1$.

If the cut contains only integer variables, we attempt to find a scaling factor that makes all coefficients close to integer. This approach uses the euclidean algorithm to compute a rational approximation
and then multiplies by the {least} common multiple ({LCM}) of all {denominators}. We impose a limit of $10^{-6}$ on the error in the rational representation and of $10^5$ on the size of the {LCM}. If this proves successful, we round the coefficients to the nearest integer and offset the \rhs to account for the difference. Consequently, we can round down the \rhs to strengthen the cut. If some of the working limits are exceeded and the approach fails, we revert to scaling to equilibrium. This is the same approach that is used in presolving for linear constraints and is detailed in \cite[Algorithm 10.2.4]{Achterberg2007}.

\subsubsection{Controlling the encoding length}
\label{sec:enc}

If scaling to integer values was not possible or not successful, {the cut was scaled to equilibrium, but the coefficients are not all integer valued. In that case, the scaled} coefficients $\alpha_i$ will often have large encoding lengths when represented as a rational number.

This is not an issue when the inequality is added to the floating-point relaxation of the LP. However, we sometimes need to (or want to) solve the rational LP relaxation of \feasP exactly, \ie using an exact LP solver. Our experiments in Section~\ref{sec:comp} show that the exact LP solver may struggle with LPs that contain cuts with coefficients of large encoding length, in some cases causing large spikes in LP solving times.

This empirical observation also has a theoretical counterpart.
The convergence analysis of the LP iterative refinement algorithm for solving
rational LPs exactly is related to the smallest possible violation of any basic solution to a rational LP.
In turn, this smallest possible violation is linked to the encoding length of
the whole problem \cite[Lemma 5]{GleixnerSteffy2020}.
This analysis suggests investigating whether the practical performance of the
iterative refinement algorithm can be improved by decreasing the encoding
length of the generated cuts.

Given a limit $M > 0$ on the size of denominators allowed in the cut, we wish to compute a relaxation $\sum_{i=1}^n\frac{n_i}{d_i}x_i \le \frac{n_0}{d_0}$ such that $ d_i, d_0 \le M$.
{Assume that we know an algorithm that can compute for any given rational number $r$ an approximation~$\frac{n}{d}$.
  Additionally, we can require either $r \le \frac{n}{d}$ or $r \ge \frac{n}{d}$, but this may decrease the quality of the approximation.
Then, for each $\alpha_i \neq 0$ in the cut we have three possible cases: 
If the corresponding variable $x_i$ is bounded only from below, \ie $x_i \ge \ell_i > -\infty$, we compute an under-approximation  $\frac{ n_i}{ d_i} \le \alpha_i$, with $ d_i \le M$. Then, we can replace $\alpha_i$ by $\frac{ n_i}{ d_i}$, as long as we add $(\frac{ n_i}{ d_i}-\alpha_i)\cdot \ell_i$ to the \rhs.
If the corresponding variable $x_i$ is bounded only from above, \ie $x_i \le u_i < \infty$, we compute an over-approximation $\frac{ n_i}{ d_i} \ge \alpha_i$, with $ d_i \le M$. Again, we replace $\alpha_i$ by $\frac{ n_i}{ d_i}$, while adding $(\frac{ n_i}{ d_i}-\alpha_i)\cdot u_i$ to the \rhs.
Finally, if the corresponding variable $x_i$ is bounded from both sides, we compute an approximation $\frac{ n_i}{ d_i} \approx \alpha_i$, with $ d_i \le M$. Then, depending on wheter $\frac{ n_i}{ d_i} \le \alpha_i$ or $\frac{ n_i}{ d_i} > \alpha_i$, we need to correct the \rhs using the lower or upper bound, respectively.}
{After performing this approximation step for each coefficient, we also compute an approximation of the \rhs. Pseudocode for this procedure is presented in Algorithm~\ref{alg:contfrac-comp}.}

\newenvironment{algocolor}{\setlength{\parindent}{0pt}
   \itshape
   \color{blue}
}{}

\begin{algorithm}
   \caption{Compute relaxation of a cut with bounded denominators}
   \label{alg:contfrac-comp}
   \begin{algorithmic}
      \STATE{\textbf{Input}: A cut $\alpha^Tx \le \beta$ over variables $x_i\in[\ell_i,u_i]$ and $M \in \Z^+$}.
      \STATE{\textbf{Output}: A relaxed cut $\sum_{i=1}^n\frac{ n_i}{ d_i}x_i \le \frac{n_0}{d_0}$ such that $ d_i, d_0 \le M$}
      \bindent
      \FOR{$i=1,\ldots,n$ with $\alpha_i \notin \Z$}
         \IF{$u_i < \infty$ and $\ell_i > -\infty$}
            \STATE{Compute approximation $\frac{ n_i}{ d_i}$ from $\alpha_i$, with $ d_i \le M$}
            \IF{$\frac{ n_i}{ d_i} \le \alpha_i$}
               \STATE{$\beta \gets \beta + (\frac{ n_i}{ d_i}-\alpha_i)\cdot \ell_i$}
            \ELSE
               \STATE{$\beta \gets \beta + (\frac{ n_i}{ d_i}-\alpha_i)\cdot u_i$}
            \ENDIF
         \ELSIF{$u_i < \infty$}
            \STATE{Compute approximation $\frac{ n_i}{ d_i}$ from $\alpha_i$ with $\alpha_i \le \frac{ n_i}{ d_i}$, and $ d_i \le M$}
            \STATE{$\beta \gets \beta + (\frac{ n_i}{ d_i}-\alpha_i)\cdot u_i$}
         \ELSIF{$\ell_i > -\infty$}
            \STATE{Compute approximation $\frac{ n_i}{ d_i}$ from $\alpha_i$ with $\alpha_i \ge \frac{ n_i}{ d_i}$, and $ d_i \le M$}
            \STATE{$\beta \gets \beta + (\frac{ n_i}{ d_i}-\alpha_i)\cdot \ell_i$}
         \ENDIF
      \ENDFOR
      \STATE{Compute approximation $\frac{n_0}{d_0}$ from $\beta$ such that $\beta \le \frac{n_0}{d_0}$, and $d_0 \le M$}
      \STATE{\textbf{Return}  $\sum_{i=1}^n\frac{ n_i}{ d_i}x_i \le \frac{\beta_n}{\beta_d}$}
      \eindent
   \end{algorithmic}
\end{algorithm} 
{In the following, we describe the} algorithm to compute these approximations of the coefficients and \rhs, which is based on \emph{continued fraction approximations} (see, \eg \cite[Section 6.1]{Schrijver1986}).
\paragraph{{Continued fraction approximations}}
For {any rational number} $r \in \Q$ we compute the continued fraction
\begin{equation*}
   [r_0;r_1,\ldots,r_t] = r_0 + \frac{1}{r_1 + \frac{1}{r_2 + \frac{1}{r_3 + \ldots}}} {\approx r}
\end{equation*}
and a sequence of \emph{convergents}
$(\frac{p_i}{q_i})_{i =1}^t$ using the following recursive formulas:
\begin{align*}
   \rho_0     & = r                       &
   r_0        & = \down{\rho_0}             \\
   \rho_{i+1} & = \frac{1}{\rho_i - r_i}  &
   r_{i+1}    & = \down{\rho_{i+1}}
   \intertext{for $i=0,1,\ldots$ as long as $\rho_{i}\not\in\Z$, and}
   p_0        & = r_0                     &
   q_0        & = 1                         \\
   p_1        & = r_0r_1+1                &
   q_1        & = r_1                       \\
   p_{i+1}    & = r_{i+1}p_{i} + p_{i-1}  &
   q_{i+1}    & = r_{i+1}q_{i} + q_{i-1}.
\end{align*}

It is known that all convergents $\frac{p_i}{q_i}$ of the continued fraction approximations are \emph{best approximations} in the sense that there exist no better approximations with a smaller denominator than $q_i$. {This has even been proven for a stronger notion of best approximation \cite[Theorem 17]{Khinchin1997}}.
However, if we impose a fixed limit $M$ on the denominator, as is the case here, we are not guaranteed that the best approximation
\[
   \arg\min\big\{ \big|\frac{n}{d} - r\big| : n\in\Z, d =1,2,\ldots,M \big\}
\]
with respect to that limit is a convergent of the continued fraction approximation.

To compute the best approximation in that sense we need to consider so-called \emph{intermediate fractions}.
If $\frac{p_i}{q_i}, \frac{p_{i+1}}{q_{i+1}}$ are consecutive convergents of a continued fraction approximation, then we define the \emph{intermediate fractions}
\begin{align*}
   p_{i+1}^j & = jp_i + p_{i-1} &
   q_{i+1}^j & = jq_i+q_{i-1}
\end{align*}
for $j = 1,\ldots,r_{i+1}$.
We can now state the desired best approximation guarantee \cite[Theorem 15]{Khinchin1997}.

\begin{Lem}
   Given $M > 0$ and $r \in \Q$, the best approximation of $r$ by a rational number with denominator at most $M$is either a convergent or an intermediate fraction of the continued fraction approximation for $r$.
\end{Lem}

\noindent

{We mention two more results from the literature that we will use to make computing the best possible approximation as efficient as possible.
   First, we note that all convergents with even index form a strictly increasing sequence, while the convergents with odd index form a strictly decreasing one \cite[Theorem 5]{Lang1966}, \ie
   \begin{equation*}
      \frac{p_0}{q_0} < \frac{p_2}{q_2} < \ldots < \frac{p_{2k}}{q_{2k}} < \ldots < r < \ldots < \frac{p_{2k+1}}{q_{2k+1}} < \ldots < \frac{p_1}{q_1}.
   \end{equation*}
Furthermore, we know that the intermediate fractions do the same \cite[Theorem 9]{Lang1966}, \ie if $i$ is even, then
   \begin{equation*}
      \frac{p_i}{q_i} < \frac{p_{i+2}^1}{q_{i+2}^1} < \ldots < \frac{p_{i+2}^{r_{i+2}}}{q_{i+2}^{r_{i+2}}} = \frac{p_{i+2}}{q_{i+2}} < \ldots\,.
   \end{equation*}
Using this monotonicity, we can formulate an algorithm to compute the best approximation of a rational number $r$ with denominator at most $M$.}

   \paragraph{{The approximation algorithm}}
{The trivial case {where the exact denominator of $r$ is no larger than} $M$ requires no work. Otherwise, we compute convergents of the continued fraction approximation for $r$ until $q_t > M$ for some $t \in \N$.
   Then, since both $\frac{p_{t-1}}{q_{t-1}}$, as well as $\frac{p_t}{q_t}$ are best approximations with respect to their denominators, we know that the best approximation with respect to $M$ is either $\frac{p_{t-1}}{q_{t-1}}$ or one of the intermediate fractions $\frac{p_t^k}{q_t^k}$ for some $k \in \{1,\ldots,r_t\}$.
   Instead of computing all intermediate fractions and choosing the best one, we know due to monotonicity that the best is the one with the largest possible denominator, \ie with $k = \down{\frac{(M - q_{t-2})}{q_{t-1}}}$.
}
Furthermore, the intermediate fractions that are best approximations are exactly the ones with $j \ge \down{r_{t}/2} + 1$.\footnote{We could not find a proof for this in the literature and have included it in the Appendix A, Lemma~\ref{lem:bestapprox}}
Hence, it is trivial to check whether the intermediate fraction or the last convergent $\frac{p_{t-1}}{q_{t-1}}$ should be chosen.

One last case to consider is when the corresponding variable is only bounded in one direction but lacks either an upper or lower bound. In that case,
we are required to find an approximation that is either not larger or not smaller than the original number $r$. This poses no additional challenge since we know that
\begin{itemize}
   \item convergents with even index are always less than or equal to $r$, and
   \item convergents with odd index are always greater than or equal to $r$,
\end{itemize}
and the same holds for intermediate fractions \cite{Lang1966}. {\wlogU, assume that we are required to compute an approximation that is not larger than $r$. Again, assume that $\frac{p_t}{q_t}$ is the first convergent such that $q_t > M$. If $t$ is odd, then we know immediately that $\frac{p_{t-1}}{q_{t-1}}$ is the best approximation, since all intermediate convergents $\frac{p_t^k}{q_t^k} > r$. If $t$ is even, then we know that $\frac{p_{t-1}}{q_{t-1}} > r$ and should therefore not be considered. Instead, we immediately choose the intermediate fraction $\frac{p_t^k}{q_t^k}$ with $k = \down{\frac{M - q_{t-2}}{q_{t-1}}}$, as described above.}

A compact algorithmic description of the whole procedure for obtaining the best approximation $\frac{p}{q} \leq r$ is presented in Algorithm~\ref{alg:contfrac}.

\begin{algorithm}
   \caption{Approximate $r = \frac{n}{d} \in \Q$ by a rational number with bounded denominator.}
   \label{alg:contfrac}
   \begin{algorithmic}
      \STATE{\textbf{Input}: $r=\frac{n}{d} \in \Q, M \in \Z^+$ }
      \STATE{\textbf{Output}: $p,q \in \Z$ with $\frac{p}{q} = {\argmin}\{|\frac{a}{b} - r| : \frac{a}{b} \leq r, b=1,2,\ldots,M\}$}
      \bindent
      \STATE{Compute $[r_0;r_1,\ldots,r_n]$ and $\frac{p_1}{q_1},\ldots,\frac{p_n}{q_n}$ with $q_{n-1} \le M$, $q_n > M$} \\
      \IF{$n$ odd, \ie $\frac{p_{n-1}}{q_{n-1}} \le r < \frac{p_n}{q_n}$}
      \STATE{$p \gets p_{n-1}$}
      \ELSE
      \STATE{$j = \down{\frac{M-q_{n-2}}{q_{n-1}}}$}
      \STATE{$p \gets jp_{n-1} + p_{n-2}$}
      \STATE{$q \gets jq_{n-1}+q_{n-2}$}
      \ENDIF
      \STATE{\textbf{Return}  $\frac{p}{q}$}
      \eindent
   \end{algorithmic}

\end{algorithm}  
\section{Generating elementary certificates}
\label{sec:veri-cuts}
Although we have proven mathematically that the proposed methods are correct and can therefore be used to solve MIPs exactly, it is unrealistic to offer a formal guarantee of correctness for an implementation of these methods. Due to the high algorithmic complexity and the sheer code size, it is neither easy for a user to verify that the implementation is indeed correct, nor is it feasible to prove correctness formally. A more realistic goal is to provide a certificate of optimality for each individually solved instance that follows a much simpler logic and can be checked and verified independently of the solving algorithm.

For MIPs, the VIPR \cite{VIPR} certificate format provides a standard for verification. It can be viewed as a tree-less encoding of (the leaves of)
a \bandb tree. To guarantee a high degree of confidence,
the format only supports the following three elementary steps:
\begin{itemize}
   \item aggregation of constraints,
   \item disjunction logic,
   \item Chvátal-Gomory cuts, \ie rounding down the \rhs of a constraint if all coefficients and variables are integer.
\end{itemize}
We give a short overview of that certificate language, for a more detailed description, see \cite{VIPR}.
A VIPR certificate consists of the problem statement, the optimal solution and objective value, followed by a derivation section that proves a lower bound or infeasibility. That derivation section is a list of constraints with proofs of their validity.  Bounds on the objective function are also seen as one type of constraint. For an \emph{aggregation}, that proof is a list of multipliers and line indices. For disjunction logic, there are two operations. A new disjunction can be introduced by printing two \emph{assumption} constraints, and then other constraints can use that assumption as part of a derivation. If at some point a constraint holds for both parts of a disjunction, it can be \emph{unsplit} and the disjunction can be discharged. At the end of the checking procedure, the checker needs to ensure that no undischarged assumptions remain. Figure~\ref{fig:viprexample} shows a short example of a VIPR certificate for an infeasible instance.  Note that the assumptions column is not provided in the certificate, but is tracked during verification.

\begin{figure}[tb]
   \begin{center}\small
      \begin{tabular}{l}\hline \\[-1ex]
         \(
         \begin{array}{rrl}
            {\bf Given}                        \\
                & x_1,x_2     & \in \mathbb{Z} \\
            C1: & 2x_1 + 3x_2 & \geq 1         \\
            C2: & 3x_1 - 4x_2 & \leq 2         \\
            C3: & -x_1 + 6x_2 & \leq 3         \\[2ex]
         \end{array}
         \)        \\ \hline\\[-1ex]
         \(
         \begin{array}{rlll}
            {\bf Derived}    &                      & {\bf Reason}                                      & {\bf Assumptions} \\
            A1:              & x_1 \leq 0           & \{ \mbox{assume} \}                                                   \\
            A2:              & x_1 \geq 1           & \{ \mbox{assume}\}                                                    \\
            A3:              & x_2 \leq 0           & \{ \mbox{assume}\}                                                    \\
            C4:              & 0 \geq 1             & \left\{ C1+ (-2)\cdot A1 + (-3)\cdot A3 \right \}
                             & A1, A3                                                                                       \\
            A4:              & x_2 \geq 1           & \{ \mbox{assume}\}                                                    \\
            C5:              & 0 \geq 1             & \left\{ \left(-\frac{1}{3}\right) \cdot C3 +
            \left(-\frac{1}{3}\right)\cdot A1 + 2\cdot A4 \right \}
                             & A1, A4                                                                                       \\
            C6:              & x_2 \geq \frac{1}{4} & \left \{ \left(-\frac{1}{4}\right)\cdot C2
            + \left(\frac{3}{4}\right)\cdot A2 \right \}
                             & A2                                                                                           \\
            C7:              & x_2 \geq 1           & \left \{ \mbox{round up } C6 \right \}            & A2                \\
            C8:              & 0 \geq 1             & \left \{ \left(-\frac{1}{3}\right)\cdot C2
            + (-1)\cdot C3 + \frac{14}{3}\cdot C7
            \right \}        & A2                                                                                           \\
            C9:              & 0 \geq 1             & \left \{ \mbox{unsplit } C4, C5 \mbox{ on }
            A3, A4 \right \} & A1                                                                                           \\
            C10:             & 0 \geq 1             & \left \{ \mbox{unsplit } C8, C9 \mbox{ on }
            A2, A1 \right \}                                                                                                \\[1ex]
         \end{array}
         \)        \\ \hline
      \end{tabular}
   \end{center}
   \caption{Certificate example for an infeasible instance \cite{VIPR}.}
   \label{fig:viprexample}
\end{figure}

For the verification of MIR cuts, let us first consider the theoretical case where all operations (aggregation, rounding, substitution) are carried out in exact arithmetic.
In that case, assuming the aggregation has already been certified, a disjunctive proof can be extracted directly from \cite{MarchandWolsey1998}.

\begin{Lem}
   \label{lem:mir}
   Given the simple case of a two-variable set
   \begin{equation}
      X = \{(w,u) \in \Z \times \R_+ : w - u \le b \},
   \end{equation}
   the MIR cut
   \begin{equation*}
      w - \frac{u}{1-f}\le \down{b}, \text{ with } f := b - \down{b}
   \end{equation*}
   is valid for $X$.
\end{Lem}

\noindent
The proof introduces a simple split disjunction:
\[
   X^1 = X \cup \{(w,u): w \le \down{b}\}
   \text{ and }
   X^2 = X \cup \{(w,u): w \ge \down{b} + 1\}.
\]
Then validity of the cut for $X^1$ is proven by aggregating
$w \le \down{b}$ and $u \ge 0$ with weights $1$ and $\frac{-1}{1-f}$, respectively.
For $X^2$ the inequalities $w \ge \down{b} + 1$ and $w-u\le b$ are aggregated with coefficients
{$-\frac{f}{1-f}$} and {$\frac{1}{1-f}$}, respectively.

For the multi-variable version needed in practice we have:

\begin{Thm}
   \label{thm:mir}
   Given $a\in\R^n$, $b\in\R$, the single-constraint set
   \begin{equation}
      X = \{(x,y^+,y^-) \in \Z^{|N|}_+ \times \R_+ \times \R_+ : a^T x + y^+ \le b + y^- \},
   \end{equation}
   and let $f = b - \down{b}, f_j = a_j - \down{a_j}$. Then the MIR cut
   \begin{equation}
      \label{eq:mircut}
      \sum_{j \in N}\Big(\down{a_j} + \frac{(f_j -f)^+}{1-f}\Big)x_j - \frac{y^-}{1-f}\le \down{b}
   \end{equation}
   is valid for $X$.
\end{Thm}

\noindent
The proof first defines two index sets $N_1 = \{j : f_j \le f\}$ and $N_2 = \{j : f_j > f\}$. Based on these, we define $w = \sum_{j \in N_1}\down{a_j}x_j + \sum_{j \in N_2} \up{a_j}x_j$, which takes the place of the single integer variable of Lemma $\ref{lem:mir}$.
For the continuous part $u = y^- + \sum_{j \in N_2} (\up{a_j}-a_j)x_j)$ is used.
Variable $y^+$ is discarded.

Therefore, to certify the correctness of inequality \eqref{eq:mircut}, we need to first print the split disjunction for $w$ as well as a proof for $u \ge 0$ to the certificate. Then we aggregate both sides of the disjunction as pointed out in the proof of Lemma~\ref{lem:mir}, unsplit both sides, and have proven validity under the restrictions of VIPR.

In practice, we perform all operations using the safe directed rounding approach discussed in the previous section. This means that we need to account for the rounding errors that were made during the solving process. A derivation for an inequality $a^Tx \le b$ is only accepted if the exact aggregation $\tilde a^Tx \le \tilde b$ \emph{dominates} $a^Tx \le b$, meaning that $a=\tilde a$ and $\tilde b \le b$.
However, $a=\tilde a $ will not necessarily hold when safe rounding was used to derive $a$, since the coefficients were not computed in exact arithmetic.

In order to handle this without creating additional overhead for generating the certificate during the solving process, we extend the certificate language by a notion of \emph{weak domination} as follows.

\begin{Def}
   An inequality $a^Tx \le b$ weakly dominates $a'^Tx \le b'$ with respect to bounds $\ell\leq x \leq u$
   if there exist coefficients $\delta_i^+, \delta_i^- \ge 0$ such that
   \begin{align*}
      a_i + \delta_i^+ - \delta_i^-                                                            & = a_i'  \text{ for all } i=1,\ldots,n, \\
      b + \sum_{i: \delta_i^+ > 0} \delta_i^+ u_i - \sum_{i: \delta_i^- > 0} \delta_i^- \ell_i & \le b'.
   \end{align*}
\end{Def}

\noindent
If a derived inequality is weakly dominating, then it is easy to post-process that derivation into a (strongly) dominating inequality by iterating through the variables, computing the coefficient differences $\delta_i = a_i' - a_i$, and adding them as additional aggregation multipliers with their corresponding bound constraints.

Note that the safe rounding technique uses precisely the ``opposite'' concept: Each arithmetic operation is accompanied by adding an overestimation of the rounding error to the \rhs. Therefore, if we print the proof according to Theorem~\ref{thm:mir} to the certificate using safe rounding, we are guaranteed to obtain a weakly dominating inequality. Figure~\ref{fig:weakexample} shows a toy example of such a weak domination and its completed strong part.

To work with this extended certificate language, we extended the certificate checking software available for VIPR\footnote{The code for VIPR is open-source and available on \cite{VIPRweb}.} by a newly created completion script
\texttt{viprcomp} that converts certificates containing weakly dominating inequalities into classic VIPR certificate files. This has the benefit of not complicating the logic used inside the proof checker itself.  Thus, confidence in the certification process remains as high as before.

Note that in theory it would be possible to perform this completion step already during the solving process, directly inside the MIP solver.
However, in practice this would negate all positive performance impact from using fast safe rounding instead of expensive symbolic computations. Performing the completion \emph{a posteriori} has further advantages. Only selected cuts must be post-processed,
and we know in advance how many cuts need to be certified; these could be post-processed independently in parallel.

One issue that this approach cannot fix is the slight floating-point inaccuracy incurred from the elimination of slack variables as described at the end of Section~\ref{subsec:mir}.
{To illustrate this issue, consider a slack variable $s_i$ that is present in the cut, belonging to some row $a_i^Tx \le b_i$. After safely applying the MIR procedure, that slack variable has some coefficient $r_i$ in the resulting cut. The same procedure applied in exact arithmetic would yield a slightly different coefficient $r_i'$. If we could compute the difference $r_i' -r_i$ in the certificate completion step, we could correct for the rounding errror using the definition of $s_i = b_i - a_i^Tx \ge 0$. However, inside the certificate, we can only observe the coefficients in the finalized cut, where the slack variable has been eliminated. Therefore, we cannot compute the correction factor just from the difference in coefficients since it is not possible to determine which rows contributed how much to the inaccuracy.}

To overcome this issue, we need to actually compute the exact coefficients {$r_i'$} for slack variables inside the solver and print
the difference {$r_i' - r_i$} to the {certificate}.
While this does increase the number of symbolic computations needed during the solve, we can limit it to only those cuts that are selected to enter the LP, and need to perform it only when certificate printing is enabled.

\begin{figure}[tb]
   \begin{center}\small
      \begin{tabular}{ll}\hline
         \(
         \begin{array}{rl}
            {\bf Given}              \\
                & x,y \in \mathbb{Z} \\
            C1: & 3x_1 - 4x_2 \leq 2 \\
            C2: & -x_1 + 6x_2 \leq 3 \\
            B3: & x_2 \le 4          \\
         \end{array}
         \) \\ \hline
         \(
         \begin{array}{rlll}
            {\bf Derived} &                                  & {\bf Reason}                                     & {\bf Comment}                  \\
            C4:           & 4\frac{2}{3}x_2 \le 3\frac{2}{3} & {\frac{1}{3}\cdot C1 + C2}                       & \text{Exact derivation}        \\
            C5:           & 5x_2 \le 5                       & {\frac{1}{3} \cdot C1 + C2}                      & \text{Incomplete}              \\
            C6:           & 5x_2 \le 5                       & {\frac{1}{3}\cdot C1 + C2 + \frac{1}{3}\cdot B3} & \text{Completed version of C5} \\
         \end{array}
         \) \\ \hline
      \end{tabular}
   \end{center}
   \caption{Toy example for weak domination. $C4$ is the actually derived constraint using the multipliers $\frac{1}{3}$, and $1$. Using the same multipliers, we might have obtained (by using some rounding operations) $C5$ instead. The reasoning for $C5$ is incomplete, but it can be converted into a proper derivation by taking into account $B3$.}
   \label{fig:weakexample}
\end{figure}

A different approach to verifying any kind of constraint that can be derived by aggregation is to solve an exact LP in the certificate completion step. Assume we are given an inequality $\alpha^Tx \le \beta$ that should be verified. Given the initial model constraints $Ax \le b$, as well as additional local constraints $A^lx \le b^l$,
\[
   \max\{\, \alpha^Tx : Ax \leq b,\; A^lx \leq b^l \,\}.
\]
If the optimal value of that exact LP is at most $\beta$, we know that $\alpha^Tx \le \beta$ is valid, and we can print a proof by aggregation using the dual multipliers of the exact LP solution. We have extended the \texttt{viprcomp} script to also contain that feature. It is not the default option to complete certificate files for two reasons. Firstly, solving an exact LP for every cut that should be verified is computationally expensive. Secondly, it is not guaranteed that the exact LP solver can prove optimality for each cut, \eg due to numerical troubles, which would render the whole certificate useless. The completion using only the variable bounds, on the other hand, is guaranteed to always work. 
\section{Computational analysis}
\label{sec:comp}

We investigate the performance of the techniques covered in the previous sections in a practical setting. Our goal is to answer three main questions.

\textit{First, how does the separation of safe GMI cuts affect the solving behaviour?}
To answer this question we evaluate safe GMI cuts both with and without rounding to small denominators as proposed in Section~\ref{sec:enc}. Besides comparing the number of solved instances, runtime, and number of \bandb nodes, we measure the \emph{gap closed} both after the root node and after the time limit.

\textit{Second, how does the effect of safe GMI cuts in the exact MIP framework compare to that of non-safe GMI cuts in the floating-point MIP setting?}
To analyze this, we first create a common baseline by disabling all features that do not exist in exact SCIP also in the floating-point version of SCIP. Within that reduced solver, we then perform experiments with and without GMI cuts and compare the results to the results with numerically safe cuts in the exact setting.

\textit{Third, what is the overhead for producing, completing, and verifying certificates?}
Here, we are particularly interested to see how the certification of cuts impacts the overhead beyond the pure branch-and-bound setting.

In total, we compare the following five configurations of \scip:
\begin{itemize}
   \item \nocuts: Exact solving mode with separation disabled.
   \item \safegmi: Exact solving mode with separation enabled, but without rounding to small denominators according to Section~\ref{sec:enc}.
   \item \enc: Exact solving mode with separation enabled \emph{and} size of denominators limited to $2^{17}$; this value proved a good tradeoff between accuracy and speed.
   \item \fpred: floating-point mode with all plugins disabled that are not available in exact solving mode and no separation.
   \item \fpcuts: floating-point mode with all plugins disabled that are not available in exact solving mode and enabled the separation of GMI cuts.
\end{itemize}

\subsection{Computational setup}

All experiments were performed on a cluster of Intel Xeon Gold 5122 CPUs with
3.6 GHz and 96 GB main memory. We use \soplexv \cite{soplex6zen} for solving all LP and exact LP subproblems. For all symbolic computations, we use
the GNU Multiple Precision Library (GMP) 6.1.2 \cite{GMP}. For symbolic presolving, we
use \papilov \cite{Papilo}; all other SCIP presolvers are disabled.

As test set we use the \miplib~2017 benchmark instances; in order to save computational effort, we exclude all those that could not be solved by the floating-point default version of \scipv within two hours. We use three random seeds for the remaining $132$ instances, making the size of our test set $396$. The time limit was set to $7200$ seconds for all experiments. Unless stated otherwise, all averages are reported as shifted geometric means with a shift of $1$ seconds and $100$ nodes, respectively.

\subsection{Dual bound improvement}

First, we analyze the \emph{gap closed} to understand how much the added safe GMI cuts help at improving dual bounds.
Given a reference bound $p$, the dual bound $d_1$ after the first LP solve, and the dual bound $d_2$ after the algorithm terminates, the gap closed is defined as $GC(p, d_1, d_2) = \frac{(d_2 - d_1)}{(p - d_1)}$.
As reference solutions, we used the best-known floating-point feasible solutions to the \miplib~2017 instances. We exclude all infeasible instances, as well as all instances where the floating-point and exact solutions do not agree (\eg because the instance is infeasible in the exact setting, but not in the floating-point setting). The results for the remaining $367$ instances are shown in Table~\ref{tbl:gapclosed}.

We observe almost identical values for the runs with \safegmi and \enc.  \enc closes on average {0.2\% more gap at the root (0.152 vs.~0.150) and 0.5\% more gap after the time limit (0.580 vs.~0.585). This shows that the further weakening of the cuts from that technique is not severe.
   Although this experiment does not present a clear picture as to which setting should be preferred}, we will see a much clearer picture in favor of \enc when looking at runtime experiments in Section~\ref{subsec:runtime}.

When comparing the impact of GMI cuts in the exact and the floating-point setting, we first observe that there is very little difference at the root node: \fpcuts closes only { $0.9\%$ more gap than \enc (0.161 vs.~0.152).
   After time limit, $11\%$ more gap is closed by enabling cuts in the floating-point setting (0.653 vs.~0.543).
In the exact setting $9.5\%$ additional gap is closed by \safegmi (0.580 vs.~0.485) and $10\%$ by \enc (0.585 vs.~0.485).}
All in all, although safely generated MIR cuts are in general weaker than floating-point cuts, our experiments show that their relative impact on the average dual performance of the respective solvers is only minimally reduced.

\begin{table}
   \centering
   \caption{Arithmetic means of gap closed at the root node and after the time limit.}
   \label{tbl:gapclosed}

\begin{tabular}{lrrrrr}
   \toprule
   {}              & \nocuts & \safegmi & \enc  & \fpred & \fpcuts \\
   \midrule
   gap closed root & -       & 0.150    & 0.152 & -      & 0.161   \\
   gap closed tlim & 0.485   & 0.580    & 0.585 & 0.543  & 0.653   \\
   \bottomrule
\end{tabular}
 \end{table}

Since mean values often do not show the complete picture, Figure~\ref{fig:clgaps} provides pairwise comparisons between different settings.
While we can also observe the very similar root node performance for \safegmi and \enc(Figure~\ref{fig:clgaps}a), there exists a significant amount of instances, where either \enc or \safegmi vastly outperforms the other after time limit (Figure~\ref{fig:clgaps}b);
some of these may be due to performance variability, which is more pronounced when comparing runs involving branching.
In Figure~\ref{fig:clgaps}c it is clearly visible that safe GMI cuts help with closing the gap in exact solving mode.
When comparing the impact in the exact and in the floating-point setting (Figure~\ref{fig:clgaps}c vs. Figure~\ref{fig:clgaps}d), we see a slightly more consistent positive impact in the floating-point setting.  The number of instances where enabling cuts leads to smaller gaps closed after time limit is larger in the exact setting.

\begin{figure}
   \centering
   \caption{Comparison of gap closed with different versions of safe GMI cuts.}
   \label{fig:clgaps}
   \includegraphics[width=.49\textwidth]{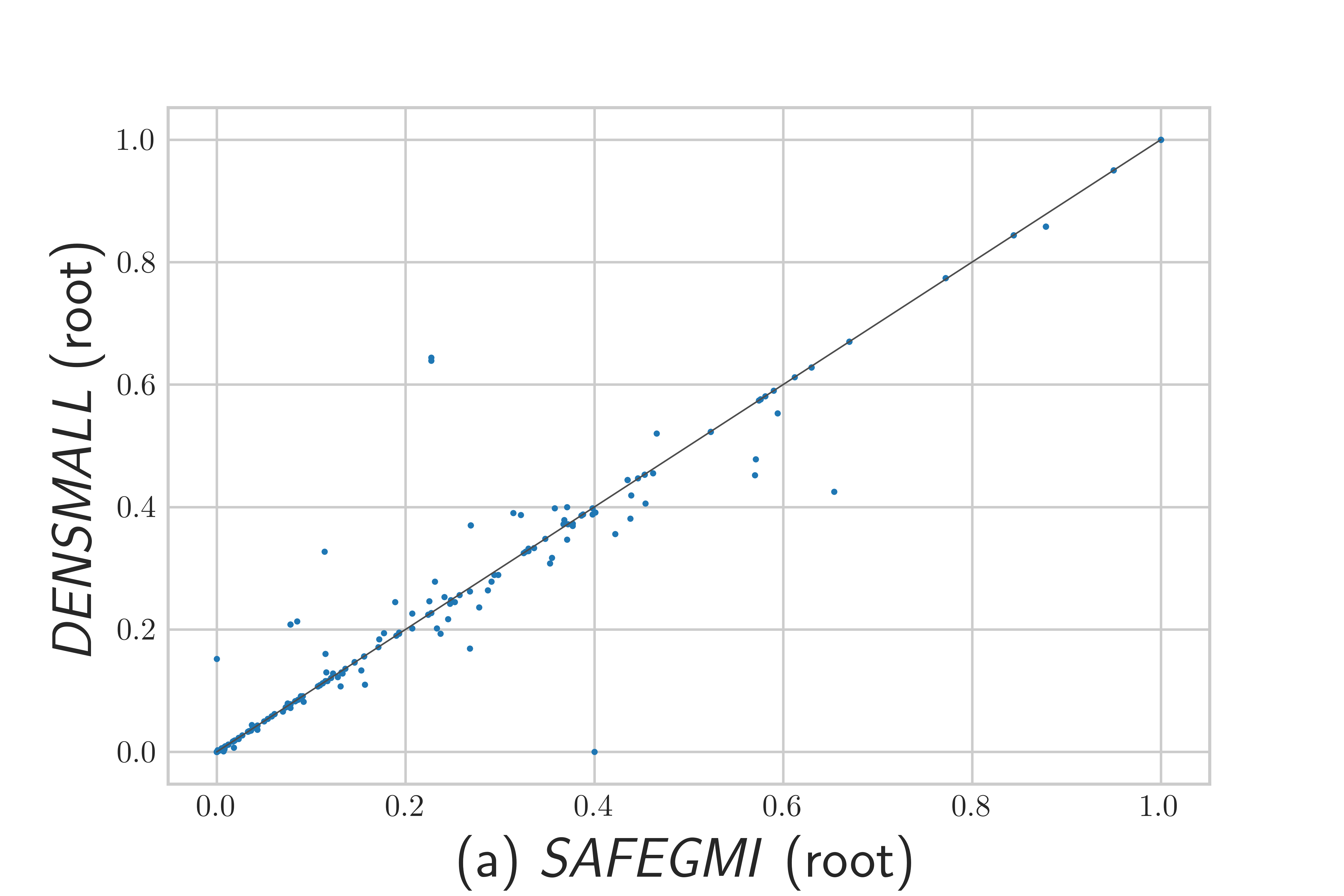}
   \includegraphics[width=.49\textwidth]{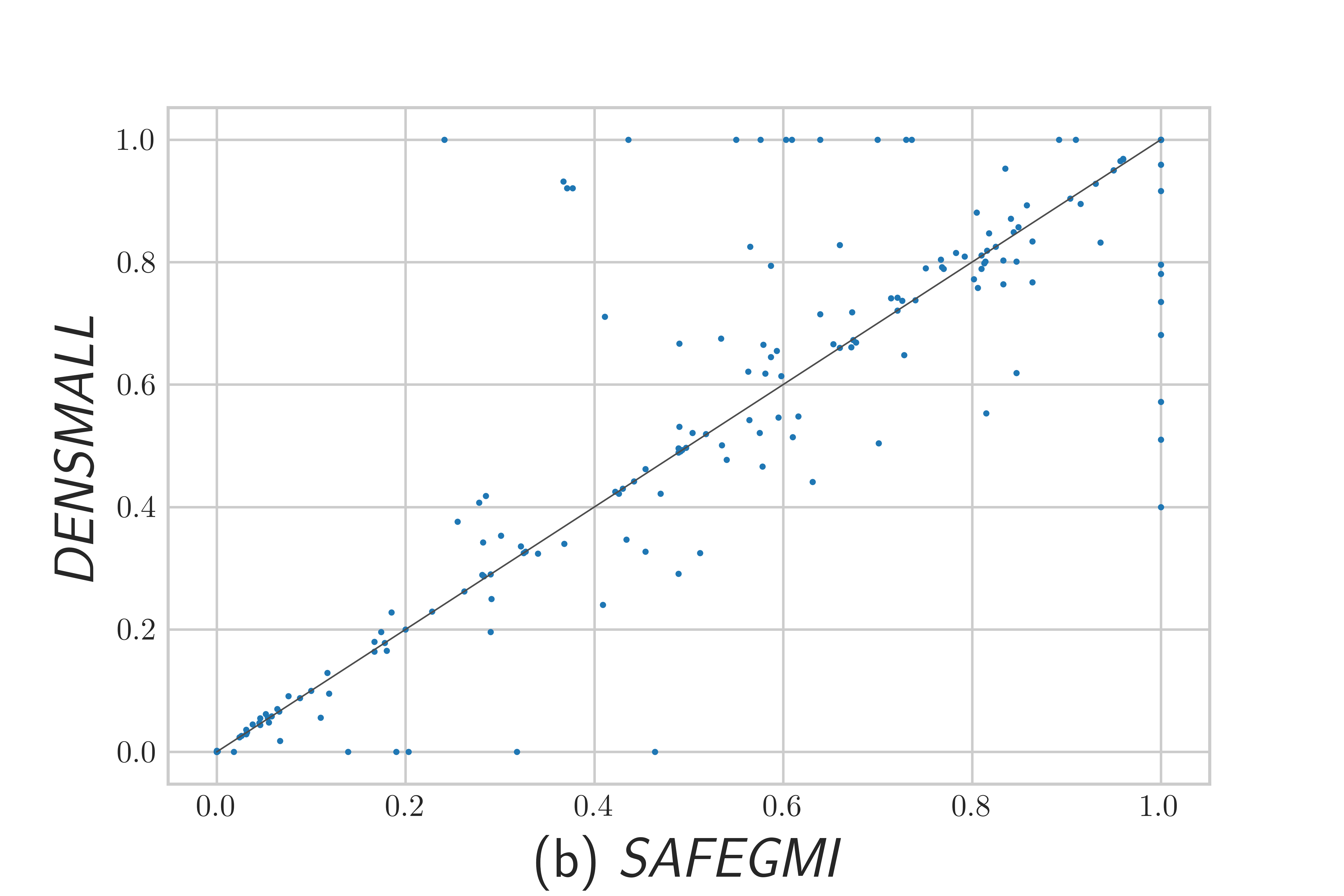} \\
   \includegraphics[width=.49\textwidth]{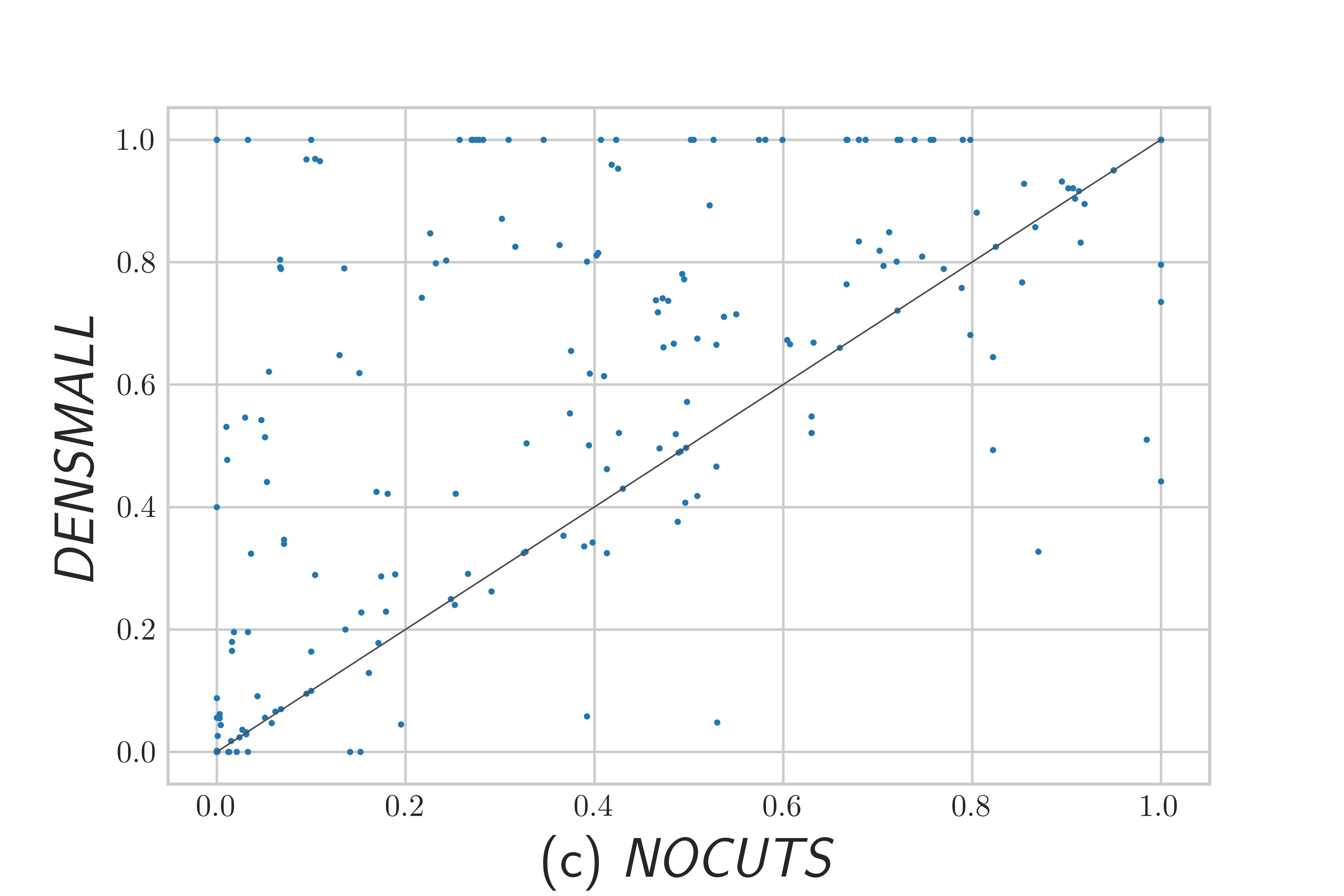}
   \includegraphics[width=.49\textwidth]{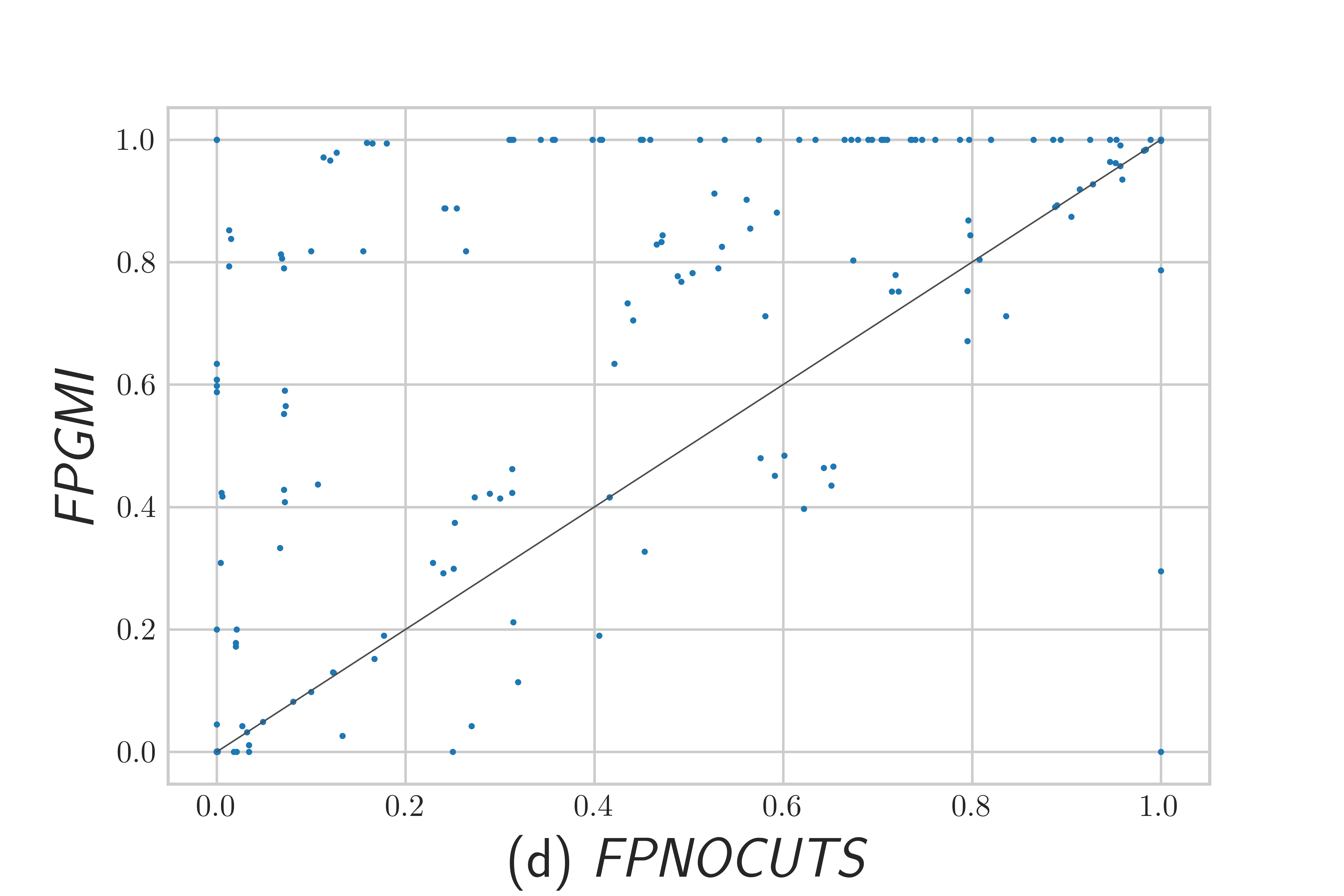}
\end{figure}

\subsection{Runtime experiments}
\label{subsec:runtime}

Table~\ref{tbl:safecuts} reports aggregated results for the number of instances solved, the number of nodes, and the running times for {the three exact settings}. Comparing the \safegmi against the \nocuts setting shows $34$ new instances could be solved, while $12$ could be solved
with \nocuts but not with \safegmi, leading to a total of $22$ more instances solved.

The number of \bandb nodes was reduced by \reduction{986.7}{2011}, and solving time was reduced by \reduction{914.26}{1121.06}. Furthermore, we observe that while the time spent in (floating-point) LP solving expectedly decreases by \reduction{216.8}{324.0} due to the smaller tree sizes, the time spent in exact LP solving increases by \increase{18.7}{7.1}.\

While the overall time spent in exact LP calls is still low, only looking at averages does not tell the full story. For many of the instances that can be solved to optimality by \nocuts but not by \safegmi, the exact LP solving time is the major bottleneck. The issue that arises within the exact LP solver for some of the subproblems is that the numerical accuracy of the underlying double-precision floating-point solver is not high enough, and that after computing the residuals and rescaling inside the iterative refinement procedure, the resulting refinement LP is numerically too difficult to solve.
This results either in large solving time spikes for some exact LPs or even for the exact LP solver to terminate without solving the problem to optimality.

In comparison to that, \enc solves $6$ more instances than \safegmi ($10$ gained, $4$ lost), with a reduction in nodes by \reduction{995.29}{2011}, and a reduction in solving time by \reduction{820.69}{1121.1} when comparing with \nocuts.
The average time spent solving exact LPs still increases by \increase{11.1}{7.1} compared to \nocuts, but most of the bottlenecks in exact LP solving disappear, and thus performance is improved.
This is although the size of the \bandb tree increases slightly compared to \safegmi.

\begin{table}
   \centering
   \caption{Comparison of safe cutting plane variants on all \miplib~2017 benchmark instances that could be solved by at least one solver. {Columns exLP and fpLP show the time spent in exact LP solving and floating-point LP solving, respectively.}}
   \label{tbl:safecuts}

\begin{tabular}{lrrlrlrr}
   \toprule
&        &         &       & \multicolumn{4}{c}{time [s]}                        \\
   \cmidrule(lr){5-8}
   setting  & solved & nodes   & (rel) & total                        & (rel) & exLP & fpLP  \\
\midrule
   \nocuts  & 130    & 20106.5 & 1.00  & 1121.06                      & 1.00  & 7.1  & 324.0 \\
   \safegmi & 152    & 9867.5  & 0.49  & 914.26                       & 0.82  & 18.7 & 216.8 \\
   \enc     & 158    & 9952.9  & 0.5   & 820.69                       & 0.73  & 11.1 & 209.5 \\
   \bottomrule
\end{tabular}

 \end{table}

This brings us to the second question, of how this measures up to the benefit from GMI cuts in the floating-point setting.
The results for \fpred and \fpcuts are presented in Table~\ref{tbl:fpcuts}. Enabling GMI cuts in that setting solves $39$ more instances, reduces the number of nodes by \reduction{12.604}{32.741}, and reduces solving time by \reduction{451.12}{695.29}. We see two main reasons for this slightly better performance in the floating-point case. First, the cuts are weakened by using the safe rounding procedure, both in their construction as well as during post-processing. We believe that this would not pose such a drastic effect in the floating-point setting, where tolerances are used to detect optimality and prune nodes.
However, since both feasibility as well as optimality tolerance is exactly zero in the exact setting, even a slight weakening of cuts may prevent nodes from being pruned or optimality from being detected. We also have to note that performance drastically worsens if the frequency of exact LP calls is increased. In this experiment, we only solve exact LPs when strictly necessary.

\begin{table}
   \centering
   \caption{Impact of GMI cuts in floating-point setting on all \miplib~2017 benchmark instances that could be solved by at least one solver.}
   \label{tbl:fpcuts}
   \begin{tabular}{lrrlrlrr}
   \toprule
           &        &         &       & \multicolumn{3}{c}{time [s]}                 \\
   \cmidrule(lr){5-7}
   setting & solved & nodes   & (rel) & total                        & (rel) & LP    \\
   \midrule
   \fpred  & 164    & 32741.6 & 1.00  & 695.29                       & 1.00  & 378.7 \\
   \fpcuts & 203    & 12604.8 & 0.38  & 451.12                       & 0.65  & 233.1 \\
   \bottomrule
\end{tabular} \end{table}

\subsection{Certificate overhead}

We measure the overhead that is introduced by enabling the verification using VIPR certificates as described in Section~\ref{sec:veri-cuts}. Since the overhead is not subject to performance variability, we only compare for a single seed, and only for our best setting, \enc. To ensure that the same number of instances is solved, we extended the time limit for the run with enabled verification to $6$ hours.
As in \cite{EiflerGleixner2022}, presolving reductions are currently not certified, so we only verify that the presolved problem has been solved correctly.  However, we do check the optimal solution for feasibility in the original problem space.

\scip does not force all generated cuts to enter the LP, but rather adds them to a storage from which efficacious cuts are selected greedily and redundant or near-orthogonal cuts are filtered. Therefore, it would be unnecessary and inefficient to write all generated cuts to the certificate file.
Instead, we only print the verification of a safe GMI cut to the certificate when the corresponding row enters the LP. To make this possible we save the aggregation information, the split information, which variable bounds are used in transformation, and the scaling factor in a hashmap.

We observe an increase of \increase{914.4}{621.5} in MIP solving time, which is very similar to what was measured in \cite{EiflerGleixner2022} for the pure branch-and-bound version.  This shows that the additional verification of cuts does not increase the proportional overhead for printing certificates.

The total overhead during solving, completion, and checking is on average $57.7\%$, which is slightly smaller than the $65.8\%$ reported in \cite{EiflerGleixner2022}, although the extra step of completing the certificates is added. This reduction shows that the effort added for verifying cutting planes is more than compensated by the effort saved due to the smaller tree sizes observed in Section~\ref{subsec:runtime}.

The time spent on completing the weakly dominating inequalities is the smallest fraction of the overhead.  The biggest share stems from bookkeeping and printing of certificates during the MIP solve. All except one instance that could be solved within a time limit of two hours without certification could be solved and successfully certified within $3$ hours; the only outlier was instance \texttt{mas74}, which took almost $8$ hours to solve and verify.

Comparing this with the overhead from other certified algorithms, \eg in recent work on symmetry breaking for pseudo-Boolean solvers \cite{BogaertsGochtMcCreeshNordstrom2022}, shows that the behaviour is quite different. In exact MIP, the overhead for printing is much larger, while the certification overhead is much smaller. The main reason for this difference is that the logic allowed in VIPR certificates is more elementary than in \cite{BogaertsGochtMcCreeshNordstrom2022}, thus putting more strain on the solver to produce certificates that are accepted.

\begin{table}
   \centering
   \caption{Overhead from producing and verifying certificates for \enc on the $49$ \miplib~2017 benchmark instances that could be solved to optimality within two hours without certification.}
   \label{tbl:cert}
   \small{
   \sisetup{round-mode=places,round-precision=1,table-number-alignment=right,table-format=2.1}
   \begin{tabular*}{\textwidth}{l@{\;\extracolsep{\fill}}llll}\toprule
      & \multicolumn{4}{c}{time [s]} \\
      \cmidrule(r){2-5}
      & solving                 & completion & checking & total \\
      \midrule
      Baseline          & 621.5                   & -          &  -       & 621.5 \\
      With verification & 914.4                   & 25.73      & 40.10    & 980.2 \\
      \midrule
      Overhead          & \overhead{914.4}{621.5} & -          & -        & \overhead{980.23}{621.5} \\
      \bottomrule
   \end{tabular*}
}
 \end{table} 

\section{Conclusion}
\label{sec:conc}

In this paper, we adapt the numerically safe Gomory mixed-integer cuts introduced by Cook et al. \cite{CookDashFukasawaGoycoolea2009} to the setting of exact rational mixed-integer programming. We identify exact LP solving as the main difficulty in making safe cutting planes performant, and show ways to overcome this difficulty by post-processing their coefficients. Using these methods, our algorithm is able to solve $21.5\%$ more instances
with a reduction of $26.8 \%$ in solving time on the \miplib~2017 benchmark test set. This is a significant improvement, although it is slightly less than what is achievable in the floating-point setting.

We conjecture that a slight weakening of the cuts from the numerically safe method together with the absence of error tolerances are the reasons for the smaller speedup in the exact setting.

We see several future research opportunities to further improve the performance of safe cutting planes. Extending the iterative refinement procedure of the exact LP solver to be able to perform precision-boosting would help with LPs that become difficult due to the separation of cutting planes. In a more straightforward direction, it might also be possible to tune separation parameters for the exact MIP setting, whereas in our experiments we use default settings of the floating-point solver to make the results as comparable as possible.

On the verification side, we show that the VIPR certificate format can be used to verify the correctness of the safe GMI cuts, without significantly increasing the overall certificate overhead.

All in all, we show that a careful implementation of numerically safe cutting planes can significantly improve the solving behaviour of an exact MIP solver. It should be pointed out that this result was achieved for one of the numerically most challenging class of cutting planes: cuts generated from the simplex tableaux. Hence, we are reasonably confident that there is room for further improvement with the addition of other types of cuts that are arithmetically easier to generate and exhibit nicer numerical properties, such as cover cuts for knapsack constraints \cite{LetchfordSouli2019}, flow covers \cite{Gu1999LiftedCI}, or zero-half cuts \cite{KosterZymolkaKutschka2009}. 
\newpage

\appendix

\section{Proof for intermediate fractions}
We consider a rational number $q$, a bound on the denominator $M > 0$, its continuous fraction
\begin{equation*}
   [r_i;r_0,\ldots,r_n] = r_0 + \frac{1}{r_1 + \frac{1}{r_2 + \frac{1}{r_3 + \ldots}}},
\end{equation*}
convergents $(\frac{p_i}{q_i})_{i=1,\ldots,n}$, and intermediate fractions $(\frac{p_i}{q_i})^j_{i=1,\ldots,n}$. Then the following holds
\begin{Lem}
   \label{lem:bestapprox}
   Exactly the intermediate fractions with $j \ge \down{\frac{r_i}{2}} + 1$ are best approximations for $q$.
\end{Lem}
\begin{proof}
   Due to the monotonicity of intermediate fractions, we only need to prove that $\frac{p_{i+1}^j}{q_{i+1}^j}$ is a best approximation for $j= \down{\frac{r_i}{2}} + 1$, and that it is not a best approximation for $\down{\frac{r_i}{2}}$.
   Furthermore, we only need to prove that  $\frac{p_{i+1}^j}{q_{i+1}^j}$ is a better approximation than $\frac{p_i}{q_i}$, \ie that
   \begin{equation*}
      \Abs{\frac{p_{i+1}^j}{q_{i+1}^j} - r} < \Abs{\frac{p_i}{q_i} - r}.
   \end{equation*}
   We know by (\cite[Theorem 9]{Khinchin1997} and \cite[Theorem 13]{Khinchin1997} that
   \begin{equation}
      \label{eq:bound}
      \frac{1}{q_i(q_{i+1}+q_{i})} < \Abs{\frac{p_i}{q_i} - r} \le \frac{1}{q_iq_{i+1}}.
   \end{equation}
   Since $\frac{p_{i+1}^j}{q_{i+1}^j}$ and $\frac{p_i}{q_i}$ lie on opposite sides of $r$, it holds that
   \begin{equation}
      \label{eq:abs}
      \Abs{\frac{p_{i+1}^j}{q_{i+1}^j} - r} = \Abs{\frac{p_{i+1}^j}{q_{i+1}^j} - \frac{p_i}{q_i}} - \Abs{\frac{p_i}{q_i} - r},
   \end{equation}
   and
   \begin{align*}
      \Abs{\frac{p_{i+1}^j}{q_{i+1}^j} - \frac{p_i}{q_i}}  = \Abs{\frac{q_i(jp_i+p_{i-1})-p_i(jq_i+q_{i-1})}{q_iq_{i+1}^j}} = \Abs{\frac{q_ip_{i-1}-p_iq_{i-1}}{q_iq_{i+1}^j}}.
   \end{align*}

   Iteratively applying the definitions of $p_i,q_i$, it is easy to see that $\Abs{q_ip_{i-1}-p_iq_{i-1}} = 1$, so we have shown
   \begin{equation}
      \Abs{\frac{p_{i+1}^j}{q_{i+1}^j} - \frac{p_i}{q_i}} = \frac{1}{q_iq_{i+1}^j}
   \end{equation}
   Using this with \eqref{eq:abs} and \eqref{eq:bound} yields
   \begin{align*}
      \Abs{\frac{p_{i+1}^j}{q_{i+1}^j} - r} & < \frac{1}{q_iq_{i+1}^j} - \frac{1}{q_i(q_{i+1}+q_{i})} = \frac{q_{i+1}+q_{i} - q_{i+1}^j}{q_iq_{i+1}^j(q_{i+1}+q_{i})} \\
                                            & = \frac{r_{i+1}q_i+q_{i-1}+q_{i} - (jq_i+q_{i-1})}{q_iq_{i+1}^j(q_{i+1}+q_{i})}
      = \frac{q^{1+r_{i+1}-j}_{i+1}-q_{i-1}}{q_iq_{i+1}^j(q_{i+1}+q_{i})}                                                                                             \\
                                            & < \frac{q^{1+r_{i+1}-j}_{i+1}}{q_iq_{i+1}^j(q_{i+1}+q_{i})}
   \end{align*}
   Setting $j = \down{r_{i+1}/2} + 1$ yields
   \begin{align*}
      \Abs{\frac{p_{i+1}^j}{q_{i+1}^j} - r} & < \frac{q_{i+1}^{\up{r_{i+1}/2}}}{q_iq_{i+1}^{\down{r_{i+1}/2}+1}(q_{i+1}+q_{i})} \\ & \le \frac{1}{q_i(q_{i+1}+q_{i})} \le \Abs{\frac{p_i}{q_i} - r}
   \end{align*}
   It still remains to show $\frac{p_{i+1}^j}{q_{i+1}^j}$ is not a best approximation for $j = \down{\frac{r_{i+1}}{2}}$.
   The idea is the same, we just use the other direction of \eqref{eq:bound}.
   \begin{align*}
      \Abs{\frac{p_{i+1}^j}{q_{i+1}^j} - r} & \ge \frac{1}{q_iq_{i+1}^j} - \frac{1}{q_iq_{i+1}} = \frac{(r_{i+1}-j)q_i}{q_iq_{i+1}^jq_{i+1}}=\frac{\up{\frac{r_{i+1}}{2}}}{(\down{\frac{r_{i+1}}{2}})q_{i+1}q_i+q_{i+1}q_{i-1}} \\
                                            & \ge \frac{\up{\frac{r_{i+1}}{2}}}{(\down{\frac{r_{i+1}}{2}})q_{i+1}q_i}  \ge \frac{1}{q_{i+1}q_i} \ge  \Abs{\frac{p_i}{q_i} - r}
   \end{align*}
\end{proof}

\section*{Acknowledgments}
We would like to thank Fabian Frickenstein for his work on the VIPR checking and completion

\bibliographystyle{abbrv}
\bibliography{status-report}

\begin{thebibliography}{10}

\bibitem{Achterberg2007}
T.~Achterberg.
\newblock {\em Constraint Integer Programming}.
\newblock PhD thesis, Technische Universit\"at Berlin, 2007.

\bibitem{AchterbergWunderling2013}
T.~Achterberg and R.~Wunderling.
\newblock Mixed integer programming: Analyzing 12 years of progress.
\newblock In M.~J{\"u}nger and G.~Reinelt, editors, {\em Facets of
  Combinatorial Optimization}, pages 449--481, 2013.

\bibitem{Applegate2009}
D.~L. Applegate, R.~E. Bixby, V.~Chvátal, W.~Cook, D.~G. Espinoza,
  M.~Goycoolea, and K.~Helsgaun.
\newblock Certification of an optimal tsp tour through 85,900 cities.
\newblock {\em Operations Research Letters}, 37(1):11--15, 2009.

\bibitem{BergBogaertsNordstromOertelVandesande2023}
J.~Berg, B.~Bogaerts, J.~Nordstr{\"o}m, A.~Oertel, and D.~Vandesande.
\newblock Certified core-guided {MaxSAT} solving.
\newblock In {\em Proceedings of the 29th International Conference on Automated
  Deduction}, 2023.

\bibitem{Bofi19}
M.~Bofill, F.~Many{\`a}, A.~Vidal, and M.~Villaret.
\newblock New complexity results for {{\L}}ukasiewicz logic.
\newblock {\em Soft Computing}, 23:2187--2197, 2019.

\bibitem{BogaertsGochtMcCreeshNordstrom2022}
B.~Bogaerts, S.~Gocht, C.~McCreesh, and J.~Nordström.
\newblock Certified symmetry and dominance breaking for combinatorial
  optimisation.
\newblock {\em Proceedings of the AAAI Conference on Artificial Intelligence},
  36(4):3698--3707, Jun. 2022.

\bibitem{Burt12}
B.~A. Burton and M.~Ozlen.
\newblock Computing the crosscap number of a knot using integer programming and
  normal surfaces.
\newblock {\em ACM Transactions on Mathematical Software}, 39(1), 2012.

\bibitem{VIPRweb}
K.~Cheung, A.~Gleixner, and D.~Steffy.
\newblock {VIPR. Verifying Integer Programming Results}.
\newblock \url{https://github.com/ambros-gleixner/VIPR} (accessed December 14,
  2022).

\bibitem{VIPR}
K.~K. Cheung, A.~Gleixner, and D.~E. Steffy.
\newblock Verifying {Integer} {Programming} {Results}.
\newblock In {\em International Conference on Integer Programming and
  Combinatorial Optimization}, pages 148--160. Springer, 2017.

\bibitem{Conforti2014}
M.~Conforti, G.~Cornuejols, and G.~Zambelli.
\newblock {\em Integer Programming}.
\newblock Springer Publishing Company, Incorporated, 2014.

\bibitem{CookDashFukasawaGoycoolea2009}
W.~Cook, S.~Dash, R.~Fukasawa, and M.~Goycoolea.
\newblock Numerically safe gomory mixed-integer cuts.
\newblock {\em INFORMS Journal on Computing}, 21:641--649, 2009.

\bibitem{CookKochSteffyetal.2013}
W.~Cook, T.~Koch, D.~E. Steffy, and K.~Wolter.
\newblock A hybrid branch-and-bound approach for exact rational mixed-integer
  programming.
\newblock {\em Mathematical Programming Computation}, 5(3):305 -- 344, 2013.

\bibitem{EiflerGleixner2022}
L.~Eifler and A.~Gleixner.
\newblock A computational status update for exact rational mixed integer
  programming.
\newblock {\em Mathematical Programming}, 2022.

\bibitem{EiflerGleixnerPulaj2022}
L.~Eifler, A.~Gleixner, and J.~Pulaj.
\newblock A safe computational framework for integer programming applied to
  chvátal's conjecture.
\newblock {\em ACM Transactions on Mathematical Software}, 48(2), 2022.

\bibitem{ElffersGochtMcCreeshNordstrom_2020}
J.~Elffers, S.~Gocht, C.~McCreesh, and J.~Nordström.
\newblock Justifying all differences using pseudo-boolean reasoning.
\newblock {\em Proceedings of the AAAI Conference on Artificial Intelligence},
  34(02):1486--1494, Apr. 2020.

\bibitem{Espinoza2006}
D.~G. Espinoza.
\newblock {\em On Linear Programming, Integer Programming and Cutting Planes}.
\newblock PhD thesis, Georgia Institute of Technology, 2006.

\bibitem{GleixnerHendelGamrathetal.2021}
A.~Gleixner, G.~Hendel, G.~Gamrath, T.~Achterberg, M.~Bastubbe, T.~Berthold,
  P.~M. Christophel, K.~Jarck, T.~Koch, J.~Linderoth, M.~L{\"u}bbecke,
  H.~Mittelmann, D.~Ozyurt, T.~Ralphs, D.~Salvagnin, and Y.~Shinano.
\newblock Miplib 2017: Data-driven compilation of the 6th mixed-integer
  programming library.
\newblock {\em Mathematical Programming Computation}, 13(3):443 -- 490, 2021.

\bibitem{GleixnerSteffy2020}
A.~Gleixner and D.~E. Steffy.
\newblock Linear programming using limited-precision oracles.
\newblock {\em Mathematical Programming}, 183:525--554, 2020.

\bibitem{stephan_gocht_2019_3548582}
S.~Gocht.
\newblock Veripb, Nov. 2019.

\bibitem{GMP}
T.~Granlund and G.~D. Team.
\newblock {\em GNU MP 6.0 Multiple Precision Arithmetic Library}.
\newblock Samurai Media Limited, London, GBR, 2015.

\bibitem{Gu1999LiftedCI}
Z.~Gu, G.~L. Nemhauser, and M.~W.~P. Savelsbergh.
\newblock Lifted cover inequalities for 0-1 integer programs: Complexity.
\newblock {\em INFORMS J. Comput.}, 11:117--123, 1999.

\bibitem{HeuleHuntWetzler2013}
M.~J. Heule, W.~A. Hunt, and N.~Wetzler.
\newblock Trimming while checking clausal proofs.
\newblock In {\em 2013 Formal Methods in Computer-Aided Design}, pages
  181--188, 2013.

\bibitem{Papilo}
A.~Hoen and L.~Gottwald.
\newblock scipopt/papilo: v2.0.0, Apr. 2022.

\bibitem{KenterEtAl2018}
F.~Kenter and D.~Skipper.
\newblock Integer-programming bounds on pebbling numbers of cartesian-product
  graphs.
\newblock In D.~Kim, R.~N. Uma, and A.~Zelikovsky, editors, {\em Combinatorial
  Optimization and Applications}, pages 681--695, 2018.

\bibitem{Khinchin1997}
A.~Y. Khinchin.
\newblock {\em Continued Fractions}.
\newblock Dover Books on Mathematics. Dover Publications, revised edition.

\bibitem{KosterZymolkaKutschka2009}
A.~M. Koster, A.~Zymolka, and M.~Kutschka.
\newblock Algorithms to separate $\{0,\frac{1}{2}\}$ chvátal-gomory cuts.
\newblock {\em Algorithmica}, 55:375--391, 2009.
\newblock http://dx.doi.org/10.1007/s00453-008-9218-7.

\bibitem{LanciaEtAl2020}
G.~Lancia, E.~Pippia, and F.~Rinaldi.
\newblock Using integer programming to search for counterexamples: A case
  study.
\newblock In A.~Kononov, M.~Khachay, V.~A. Kalyagin, and P.~Pardalos, editors,
  {\em Mathematical Optimization Theory and Operations Research}, pages 69--84,
  2020.

\bibitem{Lang1966}
S.~Lang.
\newblock {\em Introduction to diophantine approximations}.
\newblock Addison-Wesley Publishing Co., Reading, Mass.-London-Don Mills, Ont.,
  1966.

\bibitem{LetchfordSouli2019}
A.~N. Letchford and G.~Souli.
\newblock On lifted cover inequalities: A new lifting procedure with unusual
  properties.
\newblock {\em Operations Research Letters}, 47(2):83--87, 2019.

\bibitem{MarchandWolsey1998}
H.~Marchand and L.~Wolsey.
\newblock Aggregation and mixed integer rounding to solve mips.
\newblock {\em Université catholique de Louvain, Center for Operations
  Research and Econometrics (CORE), CORE Discussion Papers}, 49, 01 1998.

\bibitem{soplex6zen}
M.~Miltenberger, A.~Gleixner, T.~Koch, M.~Pfetsch, A.~Hoen, tobiasachterberg,
  F.~Schlösser, S.~Vigerske, D.~Rehfeldt, micwinx, G.~Hendel, jakobwitzig,
  M.~Besançon, D.~Steffy, bzfberth, gerald gamrath, L.~Gottwald, fserra,
  P.~Wellner, A.~Ebrahim, H.~Mulackal, J.~Nicolas-Thouvenin, stephenjmaher, and
  M.~Walter.
\newblock scipopt/soplex: v6.0.0, Apr. 2022.

\bibitem{Neumaier02safebounds}
A.~Neumaier and O.~Shcherbina.
\newblock Safe bounds in linear and mixed-integer programming.
\newblock {\em Math. Program.}, 99:283--296, 2002.

\bibitem{Orton2001}
M.~L. Overton.
\newblock {\em Numerical Computing with IEEE Floating Point Arithmetic}.
\newblock Society for Industrial and Applied Mathematics, 2001.

\bibitem{Pula20}
J.~Pulaj.
\newblock Cutting planes for families implying {F}rankl's conjecture.
\newblock {\em Mathematics of Computation}, 89(322):829--857, 2020.

\bibitem{SahraouiBendottiAmbrosio2019}
Y.~Sahraoui, P.~Bendotti, and C.~D'Ambrosio.
\newblock Real-world hydro-power unit-commitment: Dealing with numerical errors
  and feasibility issues.
\newblock {\em Energy}, 184:91--104, 2019.
\newblock Shaping research in gas-, heat- and electric- energy infrastructures.

\bibitem{Schrijver1986}
A.~Schrijver.
\newblock {\em Theory of Linear and Integer Programming}.
\newblock John Wiley \& Sons, Inc., USA, 1986.

\bibitem{SteffyWolter2013}
D.~E. Steffy and K.~Wolter.
\newblock Valid linear programming bounds for exact mixed-integer programming.
\newblock {\em INFORMS Journal on Computing}, 25(2):271--284, 2013.

\bibitem{HeuleHuntWetzler2014}
N.~Wetzler, M.~Heule, and W.~A.~H. Jr.
\newblock Drat-trim: Efficient checking and trimming using expressive clausal
  proofs.
\newblock In C.~Sinz and U.~Egly, editors, {\em Theory and Applications of
  Satisfiability Testing - {SAT} 2014 - 17th International Conference, Held as
  Part of the Vienna Summer of Logic, {VSL} 2014, Vienna, Austria, July 14-17,
  2014. Proceedings}, volume 8561 of {\em Lecture Notes in Computer Science},
  pages 422--429. Springer, 2014.

\bibitem{WilkenEtAl2000}
K.~Wilken, J.~Liu, and M.~Heffernan.
\newblock Optimal instruction scheduling using integer programming.
\newblock {\em SIGPLAN Not.}, 35(5):121--133, 2000.

\end{thebibliography}

\end{document}